\documentclass[12pt]{amsart}
\usepackage[utf8]{inputenc}

\usepackage{amsmath,amsthm,amssymb,amsfonts}
\usepackage{lipsum} 
\usepackage{txfonts} 
\usepackage{hyperref}
\usepackage{enumitem}
\hypersetup{
  pdfborder = {0 0 0}, 
  colorlinks = false    
}
\usepackage{pgf,tikz}
\usepackage{cite}
\newcommand{\IR}{{\mathbb R}}

\newcommand {\Rn}{\mathbb{R}^n}

\newtheorem{theorem}{Theorem}[section]
\newtheorem{lemma}[theorem]{Lemma}
\newtheorem{proposition}[theorem]{Proposition}

\theoremstyle{definition}

\numberwithin{equation}{section}
\begin{document}
	\title[Littlewood--Paley operators and  semigroup maximal operators
on CMO spaces]
	{Littlewood--Paley operators and  semigroup maximal operators
on CMO spaces associated to Sch\"odinger operators}
	
	\author[W. Li, Q. Lin and L. Song]{Wanjun Li, Qingze Lin  and  Liang Song*}
	\thanks{*Corresponding author}

\address{
Wanjun Li, School of Mathematics,
		Sun Yat-sen University,
		Guangzhou, 510275,
		P.R.~China.
}
 \email{liwj323@mail2.sysu.edu.cn}
	\address{
		Qingze Lin,
		Department of Mathematics,
		Shantou University,
		Shantou, 515063,
		P.R.~China.
		}
	\email{qingzelin@stu.edu.cn; linqz@alumni.sysu.edu.cn}
	
	\address{
		Liang Song,
		School of Mathematics,
		Sun Yat-sen University,
		Guangzhou, 510275,
		P.R.~China.}
	\email{songl@mail.sysu.edu.cn}


	\begin{abstract}
		Let $L=-\Delta+V$ be a Schr\"odinger operator on $\mathbb{R}^n$, where $\Delta$ is the Laplacian and $V$ satisfies the reverse Hölder inequality ${\rm RH}_q$ for some $q>n/2$. In this paper, we study the behavior of  the Littlewood-–Paley operators $s_L$ and $S_L$, as well as the semigroup maximal operator $T^*_L$, on the space ${\rm CMO}_L(\Rn)$ associated with the Schr\"odinger operator $L$. It is known from previous work that these operators are  bounded on ${\rm BMO}_L(\Rn)$. Our main result  shows that they are, in fact, mappings from ${\rm CMO}_L(\Rn)$ into itself. To prove this, we develop  several  equivalent characterizations of ${\rm CMO}_L(\Rn)$ and employ a refined decomposition that  partitions the parameter interval at $r_B\rho(x_B)$, instead of the customary  $r_B^2$ or $\rho(x_B)^2$. The new strategy allows us to overcome a key technical obstacle that arises when applying existing methods to the ${\rm CMO}_L$ setting.
	\end{abstract}
	
	\maketitle
	
	\section{Introduction}

		The space  ${\rm BMO}(\mathbb{R}^n)$ of functions of bounded mean oscillation   was introduced  by F. John and L. Nirenberg \cite{JN1961CPAM}. A locally integrable function $f$ is said to belong to ${\rm BMO}(\mathbb{R}^n)$ if
		\[
		\|f\|_{\rm BMO}=\sup\limits_B\frac{1}{|B|}\int_B|f(x)-f_B|\,dx <\infty,
		\]
		where the supremum is taken over all balls $B$ in $\Rn$ and where $f_B$ denotes the mean of $f$ over $B$. The  space ${\rm CMO}(\Rn)$ (continuous mean oscillation space) is an important closed subspace of ${\rm BMO}(\Rn)$, defined as the closure in the ${\rm BMO}$ norm of the space of  compactly supported smooth functions (see \cite{Uchiyama1978TohokuMathJ},\cite[Remark 2.6]{Neri1975Studiamath}), i.e.,  
\[  
{\rm CMO}(\Rn) = \overline{ C_0^\infty(\Rn)}^{\|\cdot\|_{\text{BMO}}}.  
\]  
It is classical that the dual space of ${\rm CMO}(\Rn)$ is the Hardy space $H^1(\Rn)$ (\cite{CW1977BAMS}), while the dual space of $H^1(\Rn)$ is ${\rm BMO}(\Rn)$ (\cite{FS1972ActaMath}).

  Littlewood--Paley operators play a crucial role  in harmonic analysis. Let us recall the definitions of the Littlewood--Paley $g$-function and area function (see \cite{Stein1970PrincetonUniversityPress}). For \( f \in L^p(\Rn)\),  
\[  
g(f)(x) = \left( \int_0^\infty |\nabla u(x, t)|^2 t\,dt \right)^{1/2} \quad {\rm and} \quad S(f)(x) = \left( \iint_{\Gamma(x)} |\nabla u(x - y, t)|^2 t^{1-n}\,dy dt \right)^{1/2}, 
\]  
where \( \Gamma(x)=\{(y, t)\in \Rn\times (0,\infty): |y - x| < t\} \), \( \nabla u(x, t) =(\partial_t u,\nabla_x u)\), and \( u(x, t) \) is the Poisson integral of $f$:  
\[  
u(x, t) = \int_{\mathbb{R}^n} P_t(y) f(x - y) \, dy, \quad \text{with } P_t(y) = c_n \frac{t}{(t^2 + |x - y|^2)^{\frac{n + 1}{2}}}.  
\]  
These two operators $g$ and $S$ are well known to be bounded on \( L^p(\Rn)\) spaces for  $1<p<\infty$.  


 Let us consider the behavior of the Littlewood--Paley operators on the ${\rm BMO}(\Rn)$ space. For \(f\in {\rm BMO}(\Rn)\), S.~Wang \cite{Wang1985book} showed that  either the Littlewood--Paley \( g \)-function satisfies \(g(f)=+\infty\) almost everywhere or $g(f)\in {\rm BMO}(\Rn)$, with $\|g(f)\|_{\rm BMO}\leq C\|f\|_{\rm BMO}$. We note that an earlier result for the Hardy--Littlewood maximal operator ${\mathcal M}$ on ${\rm BMO}(\Rn)$ was obtained in \cite{BDS1981AnnMath}, where C.~Bennett, R.~DeVore, and R.~Sharpley proved that for $f\in {\rm BMO}(\Rn)$, either ${\mathcal M}(f)\equiv +\infty$ or ${\mathcal M}(f)\in {\rm BMO}(\Rn)$, with $\|{\mathcal M}(f)\|_{\rm BMO}\leq C\|f\|_{\rm BMO}$. Later, D.~Kurtz \cite{Kurtz1987ProcAmerMathSoc} proved that an analogous result holds for the Littlewood--Paley area function $S(f)$.  Subsequently, S. Wang and J. Chen \cite{WC1990ChineseAnnMath} weakened these conditions, proving that if \( g(f) \) (or $S(f)$) is finite at a single point, then it is finite almost everywhere, and hence the operator is bounded on ${\rm BMO}(\Rn)$. We refer the reader to \cite{LY1995ApproxTheoryAppl,Sun2004NagoyaMathJ} for further results. A natural question then arises: do the operators $T\in \{{\mathcal M}, g, S\}$  map  ${\rm CMO}(\Rn)$ into itself? Recently,  H. Liu and the last two authors of the present paper \cite{LLS2025TokyoJMAth} answered this question in the negative. They showed that the same dichotomies occur on $\text{CMO}(\mathbb{R}^n)$:  for $f\in \text{CMO}(\mathbb{R}^n)$, either $T(f)\equiv +\infty$ or $T(f)\in \text{CMO}(\mathbb{R}^n)$. Moreover,  they constructed explicit examples where the operators are infinite on \( \text{CMO}(\mathbb{R}^n)\).

 The classical theory of BMO and Hardy spaces has proven remarkably successful over the past decades. Nevertheless, the standard framework fails in certain important settings-notably, in problems involving partial differential equations that require generalizations of the Laplacian. This has motivated the development of BMO space and Hardy spaces associated to a linear operator $L$, analogous to the classical theory adapted to the Laplacian. This line of research has attracted considerable attention and remains a very active topic in harmonic analysis; see, for example, \cite{DZ1997StusiaMath,DGMTZ2005MathZ,DY2005CPAM,DY2005JAMS,Auscher2007MAMS,HM2009MathAnn,HLMMY2011MAMS,AMR2008JGA,DL2013JFA,SY2016AdvMath} and references therein. 

Consider the Schr\"odinger operator
		\[
		L=-\Delta+V(x) ,
		\]
where we assume that the nonnegative potential \( V \) satisfies the reverse H\"older  condition ${\rm RH}_q(\Rn)$ for some $q>n/2$, i.e., there exists a constant $C>0$ such that
\begin{align}
    \left( |B|^{-1} \int_{B} V(x)^q \, dx \right)^{\frac{1}{q}} \leq C |B|^{-1} \int_{B} V(x) \, dx \tag{$RH_q$}
\end{align}
for all balls \( B \subset \mathbb{R}^n \).


The BMO space associated with the Schr\"odinger operator, denoted by $\mathrm{BMO}_L(\mathbb{R}^n)$, was introduced by J.~Dziubański et al.(\cite{DGMTZ2005MathZ}) and is  defined as:
\[
{\rm BMO}_L(\mathbb{R}^n) = \left\{ f \in \mathrm{BMO}(\mathbb{R}^n) : \frac{1}{|B|} \int_{B} |f(y)| \, dy < \infty \ \text{for all balls } B = B(x_B, r_B) \text{ with } r_B > \rho(x_B) \right\},
\]
with the norm 
\begin{align}\label{BMOL-norm}
\|f\|_{{\rm BMO}_L}:=\sup\limits_{B} \frac{1}{|B|}\int_B|f(x)-f_B|\,dx + \sup\limits_{B:\, r_B>\rho(x_B)} \frac{1}{|B|}\int_B|f(x)|\,dx. 
\end{align}
Here \( \rho(x) \) is the critical radius function  introduced by Z. Shen \cite{Shen1994IndianaUnivMathJ,Shen1995AnnInstFourier}, defined by
\begin{align}\label{r1}
    \rho(x) := \sup \left\{ r>0 : \frac{1}{r^{n - 2}} \int_{B(x, r)} V(h) \, dh \leq 1 \right\}. 
\end{align}

We note that \( \mathrm{BMO}_L(\mathbb{R}^n) \) is a subspace of \( \mathrm{BMO}(\mathbb{R}^n) \).  If \( V(x)\equiv 1 \), then $\rho(x)$ is a positive constant, so \( \mathrm{BMO}_{-\Delta + 1}(\mathbb{R}^n) \) coincides with the bmo space (see \cite{Goldberg1979DukeMathJ} for the  properties of bmo); if \( V(x) \equiv 0 \), then \( \rho(x) \equiv \infty \), and \( \mathrm{BMO}_{-\Delta}(\mathbb{R}^n) \) reduces to the classical \( \mathrm{BMO}(\mathbb{R}^n)\). In analogy with the proof of the John-Nirenberg inequality, the \( L^1 \)-norm appearing in \eqref{BMOL-norm} can be replaced by the \( L^p \)-norm for every \( 1 \leq p < \infty \).


This ${\rm BMO}_L$ space turns out to be the appropriate  endpoint space for the boundedness of the classical operators associated with the operator $L$. Let us now
 recall some definitions of  operators associated with $L$.   Denote \( T_s^L := e^{-sL} \) as the semigroup generated by \( L \) and  set $
Q_s^L f(x) = s \frac{d}{ds} T_s^L f(x)$. The Littlewood--Paley operators \( s_L, S_L \)  and the semigroup maximal operator \( T_L^* \) are defined as follows:
\begin{align*}
  &s_L(f)(x)= \left( \int_{0}^{\infty} |Q_s^L f(x)|^2 \frac{ds}{s} \right)^{1/2}; \\ 
  &S_L(f)(x)= \left( \int_{0}^{\infty}\!\!\! \int_{B(x,\sqrt{s})} |Q_s^L f(y)|^2 \frac{dy \, ds}{s^{\frac{n}{2} + 1}} \right)^{1/2};\\
&T_L^*f(x)= \sup_{s > 0} |T_s^L f(x)|.
\end{align*}
 

Throughout what follows,  we assume that  
\begin{align}\label{assumption of V}
   V(x)\nequiv 0 \quad   {\rm and} \quad  V\in {\rm RH}_q(\Rn) \quad {\rm for} \ \  q >n/2.  
\end{align}

Under  assumption \eqref{assumption of V}, it holds that $0<\rho(x)<\infty$ for every $x\in \Rn$ (see \cite{Shen1995AnnInstFourier}). It was proven in \cite{DGMTZ2005MathZ} that  \( \mathcal{M} \), \( T_L^* \), \( s_L \), and \( S_L \) are bounded on \( \mathrm{BMO}_L(\mathbb{R}^n) \), where $\mathcal{M}$ is the classical Hardy--Littlewood maximal operator.  This result was later extended to the Heisenberg group setting by C. Lin and H. Liu \cite{LL2011AdvMath}. It is worth noting that, in  contrast to  the Laplacian case (\( V\equiv 0 \)), such pathological infinite behavior does not occur in this setting.

It is then natural to ask:  do these operators map  \( \mathrm{CMO}_L \) into itself? To address this, we first recall the definition of \( \mathrm{CMO}_L\). For a more general operator $L$ whose associated semigroup $e^{-tL}$ satisfies a Poisson-type upper bound,  the space $\mathrm{CMO}_L$   was introduced by D. Deng et al. in \cite{DDSTY2008MichiganMathMathJ}, with the aim of developing a predual theory for the Hardy space $H^1_L$. Their definition replaces the usual integral average $f_B$ by  $e^{-r_BL}f$, and applies to a broad class of   operators.    In the special case of the Schr\"odinger operator considered in this paper, L. Ky \cite{Ky2013PotentialAnal} proved that the space ${\rm CMO}_L$   admits the following simpler equivalent characterization:   $${\rm CMO}_L(\Rn) = \overline{ C_0^\infty(\Rn)}^{\|\cdot\|_{\text{BMO}_L}}.
$$
More recently, L.~Wu and the third author of the present paper  \cite{SW2022JGeomAnal} obtained the following equivalent characterization, which provides a more convenient criterion for showing that a function belongs to ${\rm CMO}_L(\Rn)$.
\begin{proposition}[\cite{SW2022JGeomAnal}]\label{t1}
			Suppose $V\in {\rm RH}_q$ for some $q>n/2$. Let $f\in {\rm BMO}_L( \mathbb{R} ^n)$. Then $f\in {\rm CMO}_L(\IR^n)$ if and only if    $\gamma_i(f)=0$ for $1\leq i\leq5$, where
				\begin{align*}
					&\gamma_{1}(f):=\operatorname*{lim}_{a\to0}\sup_{B:r_{B}\leq a}\left(|B|^{-1}\int_{B}|f(x)-f_{B}|^{2}\,dx\right)^{1/2},\\
					&\gamma_{2}(f):=\operatorname*{lim}_{a\to\infty}\sup_{B:r_{B}\geq a}\left(|B|^{-1}\int_{B}|f(x)-f_{B}|^{2}\,dx\right)^{1/2},\\                    
					&\gamma_{3}(f):=\operatorname*{lim}_{a\to\infty}\sup_{B:B\subseteq(B(0,a))^{c}}\left(|B|^{-1}\int_{B}|f(x)-f_{B}|^{2}\,dx\right)^{1/2},\\
					&\gamma_{4}(f):=\operatorname*{lim}_{a\to\infty}\sup_{B:r_{B}\geq\max\{a,\rho(x_{B})\}}\left(|B|^{-1}\int_{B}|f(x)|^{2}\,dx\right)^{1/2},
                    \end{align*}
                    \begin{align*}
					&\gamma_{5}(f):=\operatorname*{lim}_{a\to\infty}\sup_{B:B\subseteq(B(0,a))^{c}\atop r_B\geq \rho(x_B)}\left(|B|^{-1}\int_{B}|f(x)|^{2}\, dx\right)^{1/2}.
				\end{align*}
			Here $B=B(x_B,r_B)$ and the function $\rho(x)$ is defined in \eqref{r1}.
		\end{proposition}

Most recently, X.~Han, J.~Li and L.~Wu (\cite{HLW2025NYJ}) have  shown that the Hardy--Littlewood operator  $\mathcal{M}$ and the semigroup $e^{-t\sqrt{L}}$ map ${\rm CMO}_L(\mathbb{R}^n)$ into itself. Their proof relies on  precise heat kernel estimates, Proposition \ref{t1}, and the structural properties of ${\rm CMO}_L(\mathbb{R}^n)$ established in \cite{SW2022JGeomAnal}. 
Nevertheless,  the action  of the Littlewood--Paley operators $s_L, S_L$ and the semigroup maximal operator  $T_L^*$ on  ${\rm CMO}_L(\mathbb{R}^n)$ has not yet been investigated; in fact, the argument is far from straightforward.  In the present  paper, we settle the question and formulate our main result as follows.
  \begin{theorem}\label{th6}
     Assume \eqref{assumption of V}. The operators  $s_L, S_L, T_L^*$ are mappings from  ${\rm CMO}_L(\mathbb{R}^n)$ to  ${\rm CMO}_L(\mathbb{R}^n)$. 
  \end{theorem}

  To prove Theorem \ref{th6}, we  first establish two novel characterizations of the space \( \mathrm{CMO}_L(\mathbb{R}^n) \) (see Propositions \ref{cor:equivalent character} and  \ref{CMOL-new characterization} ), which serve as foundational tools for the proof of our main results. The novelty of Proposition \ref{cor:equivalent character} lies in two aspects. One is that, we obtain  a new property of ${\rm CMO}_L$ (Lemma \ref{12}), which will be crucial  for proving the estimate \eqref{important estimate}.  The other is that, comparing with Proposition \ref{t1},  we replace the distance from the ball $B$ to the origin in $\gamma_3$ and $\gamma_5$ by the distance from its center to the origin, namely $|x_B|$.  The  advantage of this new characterization is that when $B$ is  scaled  by $k>1$, the center remains fixed, while the distance from the enlarged ball $kB$ to the origin changes dramatically. Using this new characterization will significantly simplifies our proof. 
  
  Secondly, for $f\in {\rm CMO}_L(\mathbb{R}^n)$, by Proposition \ref{CMOL-new characterization},  in order to prove $s_L(f)\in {\rm CMO}_L(\mathbb{R}^n)$,  we need to verify that $s_L(f)$ satisfies the three  conditions \eqref{f41}, \eqref{f42}, and \eqref{f43}.  The most difficult case is  checking  \eqref{f41} when $r_B\leq \rho(x_B)$. To this end, We  decompose $f=f_1+f_2+f_3$, where $f_1=(f-f_{B^*})\chi_{B^*}$, $f_2=(f-f_{B^*})\chi_{(B^*)^c}$, and  $f_3=f_{B^*}$. See \eqref{el} for details.  The terms $s_L(f_1)$ and $s_L(f_2)$ can be handled by  heat kernel  estimates. The remaining term is $s_L(f_3)=|f_{B^*}|s_L(1)$; this is  the most troublesome part, and proving its smallness (up to a constant depending only on $B$) requires some new ideas.  Previously, in the proof of $s_L(f)\in {\rm BMO}_L(\mathbb{R}^n)$, it sufficed to show that  $s_L(f_3)$ is bounded up  to a constant $C_B$. The standard strategy was to  split the integral over  $(0,\infty)$  into two pieces --- either $(0, r_B^2]$ and $(r_B^2, \infty)$ or $(0,\rho(x_B)^2]$ and $(\rho(x_B)^2,+\infty)$ --- and estimate each part separately; see, e.g.,  \cite[Theorem 8]{DGMTZ2005MathZ} and \cite[Pages 1675-1676]{LL2011AdvMath}. However, proceeding in this way makes it difficult to show that the integral over $(r_B^2, \rho(x_B)^2)$ is small.  To overcome the difficulty, we introduce a refined decomposition strategy:  we split the integral into two parts over $(0, {r_B\rho(x_B)}]$ and $[{r_B\rho(x_B)},\infty)$,  estimating the first part using the cancellation property of $tLe^{-tL}$, and the second part by the smooth estimate of $tLe^{-tL}$ (see Lemma \ref{l6}). This novel partitioning enables us to successfully  obtain the desired estimate.

 The proofs  for $S_L$ and $T_L^*$ employ a   strategy  similar to  that  for  $s_L$ in terms of key ideas, but there are moderate differences in the  details. We highlight these distinctions in our proof.



  This article is organized as follows. In section 2, we  collect some basic  facts  about  heat kernels   
 associated to Sch\"odinger operators and obtain  two new characterizations of the ${\rm CMO}_L$ space. Theorem \ref{th6} is then proved in  Sections 3-5:  Section 3 treats the operator $s_L$, Section 4 treats $S_L$, and  Section 5 treats $T^*_L$, each showing that the respective operator  maps ${\rm CMO}_L(\mathbb{R}^n)$ into itself.

  In this article, the letter $C$ or $c$ denotes a (possibly different)  constant independent of  the fundamental variables. We write  $A\approx B\,\,(resp.\,\,A\lesssim B)$ to mean that there exists a constant $C$ such that $C^{-1}A\leq B\leq CA$ (resp. $A\leq CB$).  The notation  $A\ll B$ (or $A\gg B$) indicates that $A$ is much less (respectively, much greater) than $B$. Finally, $\fint_B f(x)\,dx$ stands for $|B|^{-1}\int_Bf(x)\,dx$.

        \bigskip

        \section{Preliminaries}

\subsection{Estimates on the semigroup kernels}

In this subsection, we collect some basic  facts  of the critical function $\rho$, heat kernels   
 associated to Sch\"odinger operators  under the assumption \eqref{assumption of V}.  
 

\begin{proposition}[\cite{Shen1994IndianaUnivMathJ}]\label{p2}
				There exist constants $C>0$ and $k_0\geq 1$ such that for all $x,y\in \mathbb{R}^n$,
			\[
			C^{-1}\rho(x)\left(1+\cfrac{|x-y|}{\rho(x)}\right)^{-k_0}\leq \rho(y)\leq C\rho(x)\left(1+\cfrac{|x-y|}{\rho(x)}\right)^{\frac{k_0}{k_0+1}}.
			\]
			Consequently, $\rho(y)\leq C(\rho(x)+|x-y|)$ for all $x,y$, and   $\rho(x)\approx\rho(y)$ when  $y\in B(x,r)$ with $r\lesssim \rho(x)$.  
	\end{proposition}

  From the Feynman--Kac formula, it is known that the semigroup kernels  $k_s^L(x,y)$,  associated with  $T_s^L=e^{-sL}$, satisfy the estimates
  \begin{align}\label{6a}
      0\leq k^L_s(x,y)\leq Cs^{-\frac{n}{2}}e^{-A\frac{|x-y|^2}{s}}.
  \end{align}
The above estimates can be improved when $V>0$ satisfies the reverse H\"older condition ${\rm RH}_q$ for some $q>n/2$.
   \begin{proposition}
          [\cite{DZ2003ColloqMath}]\label{6a1}
      For any $N>0$, there exists a constant $C_N>0$ such that 
     \begin{align*}
         0\leq k_s^L(x,y)\leq C_N s^{-\frac{n}{2}}e^{-A\frac{|x-y|^2}{s}}\left(1+\cfrac{\sqrt{s}}{\rho(x)}+\cfrac{\sqrt{s}}{\rho(y)}\right)^{-N}.
     \end{align*}
  \end{proposition}
  
  \begin{proposition}
      [\cite{DZ2003ColloqMath}]\label{6a2}
    There exist constants $A, \delta>0$ such that for every $N>0$, there exists a  constant $C_N>0$ such that for $|x-x^\prime|<\sqrt{s}$, we have
     \begin{align*}
         \left|k_s^L(x,y)-k_s^L(x^\prime,y)\right|\leq C_N \left(\frac{|x-x^\prime|}{\sqrt{s}}\right)^{\delta} s^{-\frac{n}{2}}e^{-A\frac{|x-y|^2}{s}}\left(1+\cfrac{\sqrt{s}}{\rho(x)}+\cfrac{\sqrt{s}}{\rho(y)}\right)^{-N}.
     \end{align*}
  \end{proposition}
 
\begin{proposition}[\cite{DZ2002Book}]\label{6a3}  Let $k_s(x,y)$ denote the kernel of $e^{s\Delta}$. Then there exist  positive constants $C,\delta$  such that
$$\left|k_s^L(x,y)-k_s(x,y)\right|\leq Cs^{-\frac{n}{2}}e^{-A\frac{|x-y|^2}{s}}\min\left\{\left(\cfrac{\sqrt{s}}{\rho(x)}\right)^\delta,\,\,\left(\cfrac{\sqrt{s}}{\rho(y)}\right)^\delta\right\}.$$
\end{proposition}

  We also need some estimates for $Q_s^L(x,y)$, where
$Q_s^L f(x) := s \frac{d}{ds} T_s^L f(x)$.
		
		\begin{proposition}[\cite{DGMTZ2005MathZ}]\label{p1}
			There exist positive constants $A, \delta$ such that for every $N>0$, there is a constant $C_N>0$ so that 
           \begin{enumerate}[label=(\roman*), ref=\roman*]
               \item $\left|Q_s^L(x,y)\right|\leq C_N s^{-\frac{n}{2}}e^{-A\frac{|x-y|^2}{s}}\left(1+\cfrac{\sqrt{s}}{\rho(x)}+\cfrac{\sqrt{s}}{\rho(y)}\right)^{-N};$
               \item $\left|Q_s^L(x,y)-Q_s^L(x^\prime,y)\right|\leq C_N \left(\cfrac{|x-x^\prime|}{\sqrt{s}}\right)^{\delta}s^{-\frac{n}{2}}e^{-A\frac{|x-y|^2}{s}}\left(1+\cfrac{\sqrt{s}}{\rho(x)}+\cfrac{\sqrt{s}}{\rho(y)}\right)^{-N},$ \quad if\,\, $|x-x^\prime|\leq \sqrt{s};$
               \item$\left|\int_{\mathbb{R}^{n}}Q_s^L(x,y)\,dy\right|\leq C_N\left(1+\cfrac{\sqrt{s}}{\rho(x)}\right)^{-N}\left(\cfrac{\sqrt{s}}{\rho(x)}\right)^{\delta}.$
           \end{enumerate}		
           \end{proposition}

\subsection{Equivalent characterizations and properties of ${\rm CMO}_L$}

 We note that in \cite{SW2022JGeomAnal}, the authors also proved that the conditions $\gamma_2(f)=0$ and $\gamma_4(f)=0$ in Proposition \ref{t1} above can be removed. In this paper, to streamline the proof of Theorem \ref{th6},  we will provide several alternative equivalent characterizations of   ${\rm CMO}_L$,  with minor adjustments to the conditions on  $\gamma_3(f)=0$ and $\gamma_5(f)=0$.  
		\begin{lemma}		    
	\label{t2}
			If $f\in {\rm{BMO}}_L$, then $f\in {\rm CMO}_L$ if and only if,  for any $\varepsilon>0$, there exist positive constants $\sigma\ll 1$, $ R\gg 1$, and $ M\gg1$ such that
\begin{align}
			\sup_{B:\,r_{B}\leq\sigma}\left(|B|^{-1}\int_{B}|f(x)-f_{B}|^{2}\,dx\right)^{1/2}&<\varepsilon,\label{g44}\\
			\sup_{B:\,r_{B}\geq R}\left(|B|^{-1}\int_{B}|f(x)-f_{B}|^{2}\,dx\right)^{1/2}&<\varepsilon,\label{g55}\\
			\sup_{B:\,|x_B|>M}\left(|B|^{-1}\int_{B}|f(x)-f_{B}|^{2}\,dx\right)^{1/2}&<\varepsilon,\label{g66}
		\\
\sup_{B:\,r_B\geq R,\,r_{B}\geq\rho(x_{B})}\left(|B|^{-1}\int_{B}|f(x)|^{2}\,dx\right)^{1/2}&<\varepsilon,\label{g77}
		\\
\sup_{B:\,|x_B|\geq M,\,r_{B}\geq\rho(x_{B})}\left(|B|^{-1}\int_{B}|f(x)|^{2}\,dx\right)^{1/2}&<\varepsilon.\label{g88}
		\end{align}  

  \end{lemma}
		\begin{proof}
   Firstly, we restate Proposition \ref{t1} in terms of the $\varepsilon-\delta$ language:   $f\in {\rm CMO}_L(\IR^n)$ if and only if, for  every $\varepsilon>0$, there exist positive constants $\sigma\ll 1$, $R\gg 1$ and $M^\prime\gg 1$ such that the conditions 	\eqref{g44}, \eqref{g55}, \eqref{g77}	hold, along with 
 \begin{align}
 			\sup_{B:\,B\subseteq(B(0,M^\prime))^{c}}\left(|B|^{-1}\int_{B}|f(x)-f_{B}|^{2}\,dx\right)^{1/2}&<\varepsilon,\label{g6}\\
\sup_{B:\,B\subseteq(B(0,M^\prime))^{c},\,r_{B}\geq\rho(x_{B})}\left(|B|^{-1}\int_{B}|f(x)|^{2}\,dx\right)^{1/2}&<\varepsilon.\label{g8}
		\end{align}  	
        
Next, let us prove  the {\bf ``only if"} part of this proposition.  We will first show the implication:    $\eqref{g6}+\eqref{g55}\Rightarrow \eqref{g66}$. To this end, let $M$ be a sufficiently large positive number to be chosen later. We split the supremum as follows:
   \begin{align*}
    \sup_{|x_B|\geq M}\left(\fint_{B}|f(x)-f_B|^2\,dx\right)^{1/2}\leq\sup_{|x_B|\geq M\atop r_B\geq R}\left(\fint_{B}|f(x)-f_B|^2\,dx\right)^{1/2}+\sup_{|x_B|\geq M\atop r_B<R}\left(\fint_{B}|f(x)-f_B|^2\,dx\right)^{1/2}.
\end{align*}
By \eqref{g55}, 
\[
\sup_{|x_B|\geq M\atop r_B\geq R}\left(\fint_{B}|f(x)-f_B|^2\,dx\right)^{1/2}< \varepsilon.
\]
Now set $M=2M^\prime+2R$. If $x_B\geq M$ and $r_B<R$,  then  $B=B(x_B,r_B)\subset B(0,M^\prime)^c$. Hence, applying \eqref{g6} gives
\[
\sup_{|x_B|\geq M\atop r_B<R^\prime}\left(\fint_{B}|f-f_B|^2\,dx\right)^{1/2}< \varepsilon.
\]
Combining the two estimates yields \eqref{g66}. Similarly, we can prove: $\eqref{g8}+\eqref{g77}\Rightarrow \eqref{g88}$.

Lastly, the {\bf ``if"} part is easier:   since $B\subset B(0,M)^c$ implies $|x_B|\geq M$, the conclusion follows directly from the definitions. 	
		\end{proof}

The following lemma provides a useful property of ${\rm CMO}_L$, which will be applied in the proof of Lemma \ref{l4} and the   estimate \eqref{important estimate}.	
\begin{lemma}
    \label{12} If  $f\in {\rm CMO}_L(\mathbb{R}^n)$, then for any $\varepsilon>0$, there exists a constant $\sigma>0$, such that
			\begin{align}\label{ee-smallbig}
			    \sup\limits_{B:\,r_B\leq\sigma,\,r_B\geq\rho(x_B)}\bigg(\fint_B \left|f(y)\right|^2\,dy\bigg)^{1/2}< \varepsilon.
            \end{align}
		\end{lemma}
  
		\begin{proof}
Let $\varepsilon>0$ be arbitrary.  By \eqref{g88}, there exists a constant $M\gg 1$, such that for all $B=B(x_B,r_B)$  with $|x_B|\geq M$ and $r_B\geq \rho(x_B)$, there holds
			\begin{align}\label{add1}
				\left(|B|^{-1}\int_{B}|f(y)|^{2}\, dy\right)^{1/2}<\varepsilon.
			\end{align}

It remains to  handle the case when  $|x_B|<M$ and $r_B\geq \rho(x_B)$. For such balls,  Proposition \ref{p2} yields 
			\begin{align*}
				r_B\geq\rho(x_B)\geq C^{-1}\rho(0)\left(1+\frac{|x_B|}{\rho(0)}\right)^{-k_0}				\geq C^{-1} \rho(0)^{1+k_0}(\rho(0)+M)^{-k_0}=:\tilde{C}_0.
			\end{align*}
			
	Since $f\in {\rm CMO}_L(\Rn)$, it follows that $f\in L^2_{Loc}(\Rn)$, and in particular $f\in L^2(B(0,2M))$. By the absolute continuity of the integral,  there exists  $\eta>0$, such that for every measurable set $E\subset B(0,2M)$, with $|E|<\tau$,  we have 
			\[
			\int_{E}|f(y)|^2 \,dy <\tilde{C}_0^n \cdot \varepsilon^2.
			\]
 Now set $\sigma=\min{\{(\tau/c_n)^{1/n},1\}}$, where $c_n$ denotes the volume of the unit ball in $\Rn$. Then, whenever $\rho(x_B)\leq r_B<\sigma$ and $|x_B|<M$, we have $B\subset B(0,2M)$ with $|B|<\tau$. Hence

 $$\fint_{B}|f(y)|^2\,dy\leq    \tilde{C}_0^n\varepsilon^2 \cdot (c_n \tilde{C}_0^n)^{-1}\lesssim \varepsilon^2.
 $$
This, together with \eqref{add1},  completes the proof of the proposition.
		\end{proof}	

The following proposition is a direct consequence of   Lemmas \ref{t2} and \ref{12}.
\begin{proposition}\label{cor:equivalent character}
If $f\in {{\rm BMO}}_L$, then $f\in {\rm CMO}_L$ if and only if,  for any $\varepsilon>0$, there exist positive constants $\sigma\ll 1$, $ R\gg 1$, and $ M\gg1$ such that $f$ satisfies \eqref{g44}--\eqref{g88}, and \eqref{ee-smallbig}.
\end{proposition}

  To  simplify the proof of Theorem \ref{th6}, we will provide a more concise  characterization of ${\rm CMO}_L(\IR^n)$.
\begin{proposition}\label{CMOL-new characterization}
   If $f\in {\rm BMO}_L$, then $f\in {\rm CMO}_L$ if and only if  for any $\varepsilon>0$, there exist positive constants $\sigma\ll 1$,  and $ M\gg1$ such that $f$ satisfies \eqref{g44}, \eqref{g66} and \eqref{g88}.
\end{proposition}
\begin{proof}
 Necessity follows from Proposition \ref{cor:equivalent character}; it remains to prove sufficiency.
  
   Assume that $f\in {\rm BMO}_L$ and for any $\varepsilon>0$, there exist positive constants $\sigma\ll 1$,  and $ M\gg1$ such that $f$ satisfies \eqref{g44}, \eqref{g66} and \eqref{g88}.    By Proposition \ref{cor:equivalent character}, it is enough to show that, there exists a  constant  $ R\gg 1$  such that $f$ also satisfies $\eqref{g55}$ and $\eqref{g77}$. 

  Let us first verify $\eqref{g77}$.
Take $R=(10+C)\left(\rho(0)+M\right)$, where $C$ is as in Proposition \ref{p2}. Then
\begin{align*}
    \sup_{B:\,r_B\geq \max\{R,\,\rho(x_B)\}}\left(\fint_B|f(x)|^2\,dx\right)^{1/2}&\leq \sup_{B:\,|x_B|\geq M,\,r_B\geq \rho(x_B)}\left(\fint_B|f(x)|^2\,dx\right)^{1/2}+\sup_{B:\,|x_B|<M,\,r_B\geq R}\left(\fint_B|f(x)|^2\,dx\right)^{1/2}\\
    &=: I_1(f)+I_2(f).
\end{align*}

It follows from \eqref{g88} that $I_1(f)<\varepsilon$. 
Consider $I_2(f)$. Let  $|x_B|<M$ and $r_B\geq R$. Then $B(0,M)\subset B$. Let  $\tilde{x}\in B\backslash B(0,M)$ and $\tilde{r}=3(C+1)r_B$.  Then $B\subset B(\tilde{x},\tilde{r})$ and $|B(\tilde{x},\tilde{r})|=3^n(C+1)^n |B|$. By Proposition \ref{p2},
    $\rho(\tilde{x})\leq C(\rho(0)+|\tilde{x}|) \leq C(\rho(0)+r_B+M)\leq \tilde{r}$, we obtain
\begin{align*}
    \fint_B|f(x)|^2\,dx\leq \cfrac{|B(\tilde{x},\tilde{r})|}{|B|}\fint_{B(\tilde{x},\tilde{r})}|f(x)|^2\,dx\leq (C+1)^n3^n\fint_{B(\tilde{x},\tilde{r})}|f(x)|^2\,dx.
\end{align*}
Combining this with $\eqref{g88}$, along with the facts that $\tilde{r}\geq\rho(\tilde{x})$ and $|\tilde{x}|\geq M$, gives $I_2(f)\lesssim \varepsilon.$ 
Thus,  \eqref{g77} is proved.

Let $R$ be the same as above. We now prove $\eqref{g55}$. One has
\begin{align*}
     \sup_{r_B\geq R}\left(\fint_B|f(x)-f_B|^2\,dx\right)^{1/2}&\leq \sup_{|x_B|\geq M,\,r_B\geq R}\left(\fint_B|f(x)-f_B|^2\,dx\right)^{1/2}+2\sup_{|x_B|<M,\,r_B\geq R}\left(\fint_B|f(x)|^2\,dx\right)^\frac{1}{2}\\
    &=: II_1(f)+II_2(f).
\end{align*}
Observe that $II_2(f)=2I_2(f)\lesssim \varepsilon$, and from \eqref{g66},  $II_1(f)<\varepsilon$.   
\end{proof}	


		The following two lemmas, concerning  ${\rm BMO}$ and ${\rm BMO}_L$ functions respectively, are well known; for their proofs, see \cite{Stein1993book,DGMTZ2005MathZ}, respectively. 
		
   \begin{lemma}\label{F-S for BMO}
		Let $f\in {\rm BMO}(\mathbb{R}^n)$ and $\alpha>0$. Then
			\[
			r_B^\alpha \int_{(2B)^{c}}\cfrac{|f(y)-f_{B^*}|}{|y-x_B|^{n+\alpha}}\,dy\lesssim \|f\|_{{\rm BMO}}.
			\]
		\end{lemma}

        \begin{lemma}\label{l2}
			There exists $C>0$ such that, for any $f\in {\rm BMO}_L(\IR^n)$ and any ball $B(x,r)$ of \,$\mathbb{R}^n$ with $r\leq \rho(x)$, then
			\[
			\left|\cfrac{1}{\left|B(x,r)\right|}\int_{B(x, r)}f(y)\,dy \right|\leq C\left(1+\log \frac{\rho(x)}{r}\right)\|f\|_{{\rm BMO}_L}.
			\] 
		\end{lemma}

We end this section with a useful lemma  on ${\rm CMO}_L(\Rn)$, which will be often employed in the proof of our main results.
		\begin{lemma}\label{l4}
		Let $f\in {\rm CMO}_L(\mathbb{R}^n)$ and $\alpha>0$. For any $\varepsilon>0$, there exist positive constants $\sigma_1\ll1$, $R_1\gg 1$ and $M_1\gg 1$ such that for every ball $B=B(x_B,r_B)$, whenever $r_B<\sigma_1$ or $r_B>R_1$ or $|x_B|>M_1$, we have 
			\[
			r_B^\alpha \int_{(B^{*})^{c}}\cfrac{|f(y)-f(B^*,V)|}{|y-x_B|^{n+\alpha}}\,dy\lesssim \varepsilon,
			\]
		\end{lemma}
		\noindent where 
		\[ B^*=4B \quad {\rm and} \quad 
         f(B^*,V) =
		\begin{cases}
			f_{B^*} & \text{if } r_B < \rho(x_B), \\
			0 & \text{if } r_B \geq \rho(x_B).
		\end{cases}
		\]
		\begin{proof}
       By applying a dyadic decomposition, we have 
			\begin{align*}
				&r_B^\alpha \int_{(B^{*})^{c}} \cfrac{|f(y)-f(B^*,V)|}{|y-x_B|^{n+\alpha}}\,dy\\
				&\lesssim r_B^\alpha \sum_{k=1}^{\infty}\int_{2^kr_B\leq|y-x_B|<2^{k+1}r_B} |f(y)-f(B^*,V)|\left(2^kr_B\right)^{-n-\alpha}\,dy\\
				&\lesssim \sum_{k=1}^\infty 2^{-k\alpha}\fint_{B(x_B,2^{k+1}r_B)} |f(y)-f(B^*,V)|\,dy.
			\end{align*}
			
            Assume that $r_B\geq \rho(x_B)$. In this case,  $f(B^*,V)=0$. We choose a sufficiently large integer $k_1$ such that 
   \begin{align}\label{remainder term 1}
   \sum\limits_{k=k_1+1}^\infty2^{-k\alpha}\leq \varepsilon.
   \end{align}
   It follows that 
			\begin{align*}
				&\sum_{k=1}^\infty 2^{-k\alpha}\fint_{B(x_B,2^{k+1}r_B)} |f(y)-f(B^*,V)|\,dy\\
               				 &\qquad\lesssim \sum_{k=1}^{k_1}2^{-k\alpha}\fint_{B(x_B,2^{k+1}r_B)}|f(y)|\,dy
           +\sum_{k=k_1+1}^{\infty}2^{-k\alpha}\fint_{B(x_B,2^{k+1}r_B)}|f(y)|\,dy.
			\end{align*}

		The fact $\fint_{B(x_B,2^{k+1}r_B)}|f(y)|\,dy\leq \|f\|_{{\rm BMO}_L}$, together with \eqref{remainder term 1},  implies
			\begin{align}\label{sss}
		\sum_{k=k_1+1}^{\infty}2^{-k\alpha}\fint_{B(x_B,2^{k+1}r_B)}|f(y)|\,dy\,\leq \varepsilon \|f\|_{{\rm BMO}_L}.
		\end{align}	
	
    Let us estimate  $\sum_{k=1}^{k_1}2^{-k\alpha}\fint_{B(x_B,2^{k+1}r_B)}|f(y)|\,dy.$  Take $\sigma_1<2^{-k_1-1}{\sigma}$, $R_1=R$ and $M_1=M$, where $\sigma,R, M$ are the constants  appearing in Proposition \ref{cor:equivalent character}. If $r_B<\sigma_1$,  then $2^{j+1}r_B< \sigma,$ for all $j=1,\cdots,\,k_1$. If $r_B>R_1$, then $2^{j+1}r_B>R$, for all $j\in {\mathbb N}$. Combining these observations with \eqref{g77}, \eqref{g88} and Lemma \ref{12}, we deduce that    whenever $r_B<\sigma_1$, $r_B>R_1$, or $|x_B|>M$, 
			\begin{align}\label{f1}
				\sum_{k=1}^{k_1}2^{-k\alpha}\fint_{B(x_B,2^{k+1}r_B)}|f(y)|\,dy\lesssim \sum_{k=1}^{k_1}2^{-\alpha k}\varepsilon\lesssim \varepsilon.
			\end{align}   
			
			
            \noindent Combining \eqref{sss} and \eqref{f1}, we have proved Lemma \ref{l4} for the case $r_B\geq \rho(x_B)$.

            \smallskip
            
 			Assume that $r_B<\rho(x_B)$. In this case, $f(B^*,V)=f_{B^*}$. One can calculate
			\begin{align*}
				&\sum_{k=1}^\infty 2^{-k\alpha}\fint_{B(x_B,2^{k+1}r_B)} |f(y)-f(B^*,V)|\,dy\\
				&\quad\lesssim \sum_{k=1}^\infty2^{-k\alpha}\left(\fint_{B(x_B,2^{k+1}r_B)}|f(y)-f_{2^{k+1}B}|\,dy+|f_{B^*}-f_{2^{k+1}B}|\right).
			\end{align*}

   Note that 
   \begin{align*}
       \left|f_{B^*}-f_{2^{k+1}B}\right|\leq \sum_{j=1}^{k}|f_{2^jB}-f_{2^{j+1}B}|
       \lesssim k\sup_{1\leq j\leq k}\fint_{2^{j+1}B}|f(y)-f_{2^{j+1}B}|\,dy.
   \end{align*}
  Denote $I_k:=\sup_{1\leq j\leq k}\fint_{2^{j+1}B}|f(y)-f_{2^{j+1}B}|\,dy$. 
   Then  
   \begin{align*}
      \sum_{k=1}^\infty 2^{-k\alpha}\fint_{B(x_B,2^{k+1}r_B)} |f(y)-f(B^*,V)|\,dy\lesssim \sum_{k=1}^\infty 2^{-k\alpha}(k+1)I_k.
   \end{align*}
 
 Choose a sufficiently large integer $k_2$ such that $\sum\limits_{k=k_2+1}^\infty2^{-k\alpha}(k+1)\leq \varepsilon$. Together with the estimate $I_k\leq \|f\|_{{\rm BMO}_L}$, this yields 	\begin{align}\label{4.5}	\sum_{k=k_2+1}^{\infty}2^{-k\alpha}(k+1)I_k\,\leq \varepsilon \|f\|_{{\rm BMO}_L}.
			\end{align}
			
		It remains to estimate $\sum_{k=1}^{k_2}2^{-k\alpha}(k+1)I_k$. 	Set $\sigma_1<2^{-k_2-1}{\sigma}$, $R_1=R$ and $M_1=M$, where $\sigma, R, M$ are as in Proposition \ref{cor:equivalent character}. If $r_B<\sigma_1$, then $2^{j+1}r_B\leq \sigma$, for all $j=1,\,2\,,\cdots,\,k_2$. If $r_B>R_1$, then $2^{j+1}r_B>R$, for all $j\in {\mathbb N}$. These facts, together with  (\ref{g44})-\eqref{g66}, imply that whenever $r_B<\sigma_1, r_B>R_1$, or $|x_B|>M_1$,  we have 
			\begin{align}\label{the front term}
				\sum_{k=1}^{k_2}2^{-k\alpha}(k+1)I_k\lesssim \sum_{k=1}^{k_2}2^{-\alpha k}(k+1)\varepsilon\lesssim \varepsilon.
			\end{align}

			Combining (\ref{4.5}) and (\ref{the front term}), we have proved Lemma \ref{l4} in the case where  $r_B<\rho(x_B)$.          		
		\end{proof}

\bigskip

\section{The mapping $s_L:\ {\rm CMO}_L\to {\rm CMO}_L$}

  Suppose that $f\in {\rm CMO}_L(\mathbb{R}^n)$.
It has been shown in \cite[Theorem 8]{DGMTZ2005MathZ} that $s_L(f)\in {\rm BMO}_L$ and $\|s_L(f)\|_{{\rm BMO}_L}\lesssim\|f\|_{{\rm BMO}_L}$. To prove that $s_L(f)\in {\rm CMO}_L$\,,  by Proposition \ref{CMOL-new characterization},  it suffices to show that for any $\varepsilon>0$, there exist positive constants $\tilde{\sigma}\ll1$ and $\tilde{M}\gg1$ such that $s_L(f)$ satisfies the following three conditions,
		\begin{align}
	\sup\limits_{B: \,r_B\leq \tilde{\sigma}}\left(\fint_{B}\left|s_L(f)(x)-(s_L(f))_B\right| ^2\,dx\right)^{\frac{1}{2}}&\lesssim \varepsilon,\label{f41}\\
	\sup\limits_{B:\,|x_B|\geq \tilde{M}}\left(\fint_{B}\left|s_L(f)(x)-(s_L(f))_B\right| ^2\,dx\right)^{\frac{1}{2}}&\lesssim\varepsilon,\label{f42}	\\			
 \sup\limits_{B:\,|x_B|\geq \tilde{M},\,r_B\geq \rho(x_B) }\left(\fint_{B}\left|s_L(f)(x)\right| ^2\,dx\right)^{\frac{1}{2}}&\lesssim\varepsilon, \label{f43}
\end{align}
where $B=B(x_B,r_B)$ denotes the ball of radius $r_B$ centered at $x_B$. 

        \bigskip
  
\textbf{Proof of  \eqref{f41}.} 
  By a standard argument, to  prove \eqref{f41}, it suffices to prove 
            \begin{align}\label{r<a and r<rho}
\sup\limits_{B: \,r_B\leq \tilde{\sigma},\, r_B<\rho(x_B)}\left(\fint_{B}\left|s_L(f)(x)-C(B)\right| ^2\,dx\right)^{\frac{1}{2}}\lesssim \varepsilon,
            \end{align}
and
\begin{align}\label{r<a and r>rho}
\sup\limits_{B: \,r_B\leq \tilde{\sigma},\, r_B\geq \rho(x_B)}\left(\fint_{B}\left|s_L(f)(x)\right| ^2\,dx\right)^{\frac{1}{2}}\lesssim \varepsilon,
            \end{align}
where  $C(B)$ in \eqref{r<a and r<rho} is a constant depending only on the ball $B$ and is to be determined later.

\medskip

  {\bf Let us first prove  \eqref{r<a and r<rho}.} Denote $B=B(x_B,r_B)$. In this case,   {\it we always assume    $r_B<\rho(x_B)$}.
 
 Define $B^*=B(x_B,4r_B)$ and set			\begin{equation}\label{el}
				f=(f-f_{B^*})\chi_{B^*}+(f-f_{B^*})\chi_{(B^*)^c}+f_{B^*}=:f_1+f_2+f_3.
			\end{equation}
			
Let us define  $C(B)$ as follows:
\begin{align}\label{s:C(B)}
C(B):=\left\|Q^L_sf_2(x_B)\chi_{(r_B^2,\,\infty)}(\cdot)+Q^L_sf_3(x_B)\chi_{(r_B\rho(x_B),\,\infty)}(\cdot)\right\|_{L^2(ds/s)}.
\end{align}

\noindent It is not difficult to show that $C(B)\leq C_1(B)+C_2(B)<\infty$, where
$$
C_1(B):=\left(\int_{r_B^2}^{\infty}|Q_s^Lf_2(x_B)|^2\frac{ds}{s}\right)^{\frac{1}{2}} \quad {\rm and} \quad
C_2(B):=\left(\int_{r_B\rho(x_B)}^{\infty}\left| Q_s^Lf_3(x_B)\right|^2\frac{ds}{s}\right)^{\frac{1}{2}}.
$$
In fact, 
by (i) of Proposition \ref{p1}, one has
  \begin{align}\label{s:C_1(B)}
      C_1(B)^2
      &\lesssim \int_{r_B^2}^\infty \left\{ \int_{\Rn}s^{-\frac{n}{2}}e^{-A\frac{|x_B-y|^2}{s}}\left(\cfrac{\sqrt{s}}{\rho(x_B)}\right)^{-2}|f_2(y)|\,dy
 \right\}^2\,\cfrac{ds}{s}\nonumber\\
 &\lesssim \int_{r_B^2}^\infty \left(\frac{\rho(x_B)}{\sqrt{s}}\right)^4 \left(\frac{\sqrt{s}}{r_B}\right)^2\left\{ \int_{(B^*)^c} \frac{r_B}{|x_B-y|^{n+1}}|f(y)-f_{B^*}|\,dy
 \right\}^2\,\cfrac{ds}{s}\\
  &\lesssim \left(\frac{\rho(x_B)}{r_B}\right)^4\|f\|_{{\rm BMO}}^2,\nonumber
  \end{align}
  where in the last inequality above we have used Lemma \ref{F-S for BMO}.  Similarly,
 \begin{align}\label{s:C_2(B)}
      C_2(B)^2
      &\lesssim |f_{B^*}|^2\int_{r_B\rho(x_B)}^\infty \left\{ \int_{\Rn}s^{-\frac{n}{2}}e^{-A\frac{|x_B-y|^2}{s}}\left(\cfrac{\sqrt{s}}{\rho(x_B)}\right)^{-2}\,dy
 \right\}^2\,\cfrac{ds}{s}\nonumber\\
 &\lesssim |f_{B^*}|^2   \int_{r_B\rho(x_B)}^\infty \left(\frac{\rho(x_B)}{\sqrt{s}}\right)^4\, \cfrac{ds}{s}\\
 &\lesssim  |f_{B^*}|^2 \left(\frac{\rho(x_B)}{r_B}\right)^2.\nonumber
\end{align}
The above estimates, together with the fact that
 ${\rm BMO}_L\subset {\rm BMO} \subset L^1_{loc}$, imply that $C(B)<\infty$.

For any $x\in B$,  we apply Minkowski's inequality to get
\begin{equation}\label{f3}
\begin{aligned}
    |s_L(f)(x)-C(B)|&=\left| \left\|Q_s^Lf_1(x)+Q_s^Lf_2(x)+Q_s^Lf_3(x)\right\|_{L^2(ds/s)}-C(B)\right|\\
    &\leq \left\|Q_s^Lf_1(x)\right\|_{L^2(ds/s)}+\left\|Q_s^Lf_2(x)-Q_s^Lf_2(x_B)\chi_{(r_B^2,\,\infty)}(\cdot)\right\|_{L^2(ds/s)}\\
    &\quad+\left\|Q_s^Lf_3(x)-Q^L_sf_3(x_B)\chi_{(r_B\rho(x_B),\,\infty)}(\cdot)\right\|_{L^2(ds/s)}\\
    &=:I_1(x)+I_2(x)+I_3(x).
\end{aligned}
\end{equation}
Hence,
\begin{align}\label{s-function}
\left(\fint_B |s_L(f)(x)-C(B)|^2\,dx\right)^{1/2}\leq \sum_{j=1}^3\left(\fint_B (I_j(x))^2 \,dx\right)^{1/2}.
\end{align}

By the $L^2$ boundedness of $s_L$, we  obtain
				\begin{align}\label{ee-I1}
				\fint_{B}|I_1(x)|^2 dx=\fint_{B}|s_L(f_1)(x)|^2 dx\lesssim {|B|}^{-1}\int_{\mathbb{R}^n}|f_1(x)|^2 dx
				\lesssim\cfrac{1}{|B^*|}\int_{B^*}|f(x)-f_{B^*}|^2 dx.	
				\end{align}
           Let $\tilde{\sigma}_1={\sigma}/4$, where $\sigma$ is as in Proposition \ref{cor:equivalent character}.  By  (\ref{g44}), we have that  whenever $r_B\leq \tilde{\sigma}_1$, there holds           \begin{align}\label{I1-estimate}
            \cfrac{1}{|B|}\int_{B}|I_1(x)|^2 dx\lesssim \fint_{B^*}|f(x)-f_{B^*}|^2\,dx\leq \varepsilon^2. 
            \end{align}

        We now estimate  $I_2(x)$ for $x\in B$. For such $x$, we decompose
			\begin{align}	
   I_2(x)&\leq \left(\int_{0}^{r_B^2}\left|Q_s^Lf_2(x)\right|^2\frac{ds}{s}\right)^{\frac{1}{2}}
+\left(\int_{r_B^2}^{\infty}\left|Q_s^Lf_2(x)-Q_s^Lf_2(x_B)\right|^2\frac{ds}{s}\right)^{\frac{1}{2}}\notag\\
				&=: I_{21}(x)+I_{22}(x).\label{41}
			\end{align} 
		Using  (i) of Proposition \ref{p1} together with the observation that for any $x\in B$ and  $y\in (B^*)^c$,
        $$
        |y-x_B|\leq |y-x|+|x-x_B|\leq |y-x|+r_B\leq 2|y-x|,
        $$  
    we obtain, for every $x\in B$,
			\begin{align}\label{I21}
				(I_{21}(x))^2
				&\lesssim 
				\int_{0}^{r_B^2}\left(\int_{\Rn}s^{-\frac{n}{2}}e^{-A\frac{|x-y|^2}{s}}|f_2(y)|\,dy\right)^2\frac{ds}{s}\nonumber\\
				&\lesssim
				\int_{0}^{r_B^2}\left(\frac{\sqrt{s}}{r_B}\right)^2\frac{ds}{s}\ \left(\int_{(B^*)^c}\cfrac{r_B}{|y-x_B|^{n+1}} |f(y)-f_{B^*}|\,dy\right)^2\\
                &\lesssim
   \left(\int_{(B^*)^c}\cfrac{r_B}{|y-x_B|^{n+1}} |f(y)-f_{B^*}|\,dy\right)^2. \nonumber
			\end{align}
						
		We proceed to estimate $I_{22}(x)$ for $x\in B$. Observe that $|x-x_B|<r_B\leq \sqrt{s}$ for $x\in B$ and $e^{-|z|}\lesssim (1+|z|)^{-\alpha_1}\leq |z|^{-\alpha_1}$ for any $\alpha_1>0$. We apply (ii) of Proposition \ref{p1} to obtain, for any $x\in B$,
			\begin{align}\label{I22}
				(I_{22}(x))^2
				&\lesssim 
				\int_{r_B^2}^{\infty}\left[\int_{\Rn}\left(\frac{|x-x_B|}{\sqrt{s}}\right)^{\delta}s^{-n/2}e^{-A\frac{|x_B-y|^2}{s}}|f_2(y)|\,dy\right]^2\frac{ds}{s}\nonumber\\
				&\lesssim
				\int_{r_B^2}^{\infty}\left(\frac{r_B}{\sqrt{s}}\right)^{2\delta}\left(\int_{\Rn} \frac{(\sqrt{s})^{\alpha_1}}{|x_B-y|^{n+\alpha_1}}|f_2(y)|\,dy\right)^2\frac{ds}{s}\nonumber\\
                &\lesssim
\int_{r_B^2}^{\infty}\left(\frac{r_B}{\sqrt{s}}\right)^{2(\delta-\alpha_1)}\frac{ds}{s}\left(\int_{(B^*)^c} \frac{{r_B}^{\alpha_1}}{|x_B-y|^{n+\alpha_1}}|f(y)-f_{B^*}|\,dy\right)^2\\
&\lesssim \left(\int_{(B^*)^c} \frac{{r_B}^{\alpha_1}}{|x_B-y|^{n+\alpha_1}}|f(y)-f_{B^*}|\,dy\right)^2\nonumber
			\end{align}	
			where we take $0<\alpha_1<\delta$.
            
            Combing \eqref{I21} and \eqref{I22}, we have shown that 
\begin{align}\label{ee-I2}
\left(\fint_B |I_2(x)|^2\,dx\right)^{1/2}\lesssim \int_{(B^*)^c} \frac{{r_B}^{\alpha_2}}{|x_B-y|^{n+\alpha_2}}|f(y)-f_{B^*}|\,dy,
\end{align}
where $\alpha_2:=\min\{\alpha_1,1\}$.    Let $\tilde{\sigma}_2=\sigma_1$, where $\sigma_1$ is as in Lemma \ref{l4}. By applying  Lemma \ref{l4} to \eqref{ee-I2}, we obtain that when  $r_B<\tilde{\sigma}_2$, 
            \begin{align}\label{I2-estimate}
            \left(\fint_B |I_2(x)|^2\,dx\right)^{1/2}\lesssim \varepsilon.
            \end{align}

\smallskip

We still need to estimate $I_3(x)$ for $x\in B$, which is a delicate part  and calls for extra care. We begin by establishing  Lemma \ref{l6}. Note that  the factor $\left(\frac{r_B}{\rho(x_B)}\right)^\eta$ appearing in \eqref{I3} is crucial for showing that  $I_3(x)$ is small.
			\begin{lemma}\label{l6}
                  For any $x\in B$, if $r_B<\rho(x_B)$, then there exists a constant $\eta>0$ such that 
				\begin{align}\label{I3}
				    I_3(x)\lesssim |f_{B^*}|\left(\frac{r_B}{\rho(x_B)}\right)^\eta.
                \end{align}
			\end{lemma}
			\begin{proof}

                Let $x\in B$. One can compute
				\begin{align*}
					I_3(x)&\leq|f_{B^*}| \left(\int_{0}^{r_B\rho(x_B)}\left|Q_s^L(1)(x)\right|^2\frac{ds}{s}\right)^{\frac{1}{2}}
					+|f_{B^*}| \left(\int_{r_B\rho(x_B)}^{\infty}\left|Q_s^L(1)(x)-Q_s^L(1)(x_B)\right|^2\frac{ds}{s}\right)^{\frac{1}{2}}\\
					&=:|f_{B^*}|(I_{31}(x)+I_{32}(x)).
				\end{align*}
				
				For the term $I_{31}$, we apply  (iii) of Proposition \ref{p1} together with the fact that $\rho(x)\approx \rho(x_B)$ whenever $x\in B$ and $r_B<\rho(x_B)$, to show that for any $x\in B$,
				\begin{align*}
					I_{31}(x)&\lesssim \left(\int_0^{r_B\rho(x_B)}\left(\cfrac{\sqrt{s}}{\rho(x_B)}\right)^{2\delta}\frac{ds}{s} \right)^\frac{1}{2}
					\lesssim \left( \frac{r_B}{\rho(x_B)} \right)^{\frac{\delta}{2}}.
				\end{align*}

                For the term $I_{32}$, we note that $|x-x_B|<r_B<\sqrt{r_B\rho(x_B)}\leq\sqrt{s}$ when $s\geq r_B\rho(x_B)$. 
 Then we may use (ii) of Proposition \ref{p1} to obtain that for any $x\in B$,
				\begin{align*}
					I_{32}(x)&\lesssim \left\{\int_{r_B\rho(x_B)}^{\infty} \left(\cfrac{|x-x_B|}{\sqrt{s}}\right)^{2\delta} \bigg( \int_{\mathbb{R}^{n}}s^{-\frac{n}{2}}e^{-A\frac{|x-y|^2}{s}} \,dy \bigg)^2\cfrac{ds}{s}  \right\}^{\frac{1}{2}}\\
					&\lesssim \biggl(\int_{r_B\rho(x_B)}^\infty \left(\frac{r_B}{\sqrt{s}}\right)^{2\delta}\frac{ds}{s}\biggr)^{\frac{1}{2}}\\
					&\lesssim \left(\cfrac{r_B}{\rho(x_B)}\right)^{\frac{\delta}{2}}.
				\end{align*}
    We complete the proof of this lemma. 
			\end{proof}	


		With Lemma \ref{l6} at our disposal, we can now proceed to  complete the estimate  $I_3(x)$.	  By Lemmas \ref{l6} and  \ref{l2}, we have that for  $x\in B$ with $r_B<\rho(x_B)$,
			\begin{align*}
				I_3(x)&\lesssim \left(\cfrac{r_B}{\rho(x_B)}\right)^\eta\bigg[1+\log\left(\frac{\rho(x_B)}{r_B}\right)\bigg]\|f\|_{{\rm BMO}_L}\\
                &\lesssim  \left(\cfrac{r_B}{\rho(x_B)}\right)^{\eta/2}\|f\|_{{\rm BMO}_L}.
                \end{align*}

               If $ \rho(x_B)\geq \sigma/2$, where $\sigma$ is as in Proposition \ref{cor:equivalent character}. Let $\tilde{\sigma}_3=\Big(\frac{\varepsilon}{\|f\|_{{\rm BMO}_L}}\Big)^{2/\eta}\sigma$. Then we have that for $x\in B$,  
                \begin{align}\label{I3-1}
                I_3(x)\lesssim  \left(\cfrac{r_B}{\sigma}\right)^{\eta/2}\|f\|_{{\rm BMO}_L}   \leq\varepsilon,  \quad {\rm when} \ \ r_B<\tilde{\sigma}_3.
			    \end{align}

          Consider the remaining case  $\rho(x_B)<\sigma/2$.  Since $r_B<\rho(x_B)$, we may  assume that $2^{l-1}r_B\leq  \rho(x_B)<2^{l}r_B,$  for some $l\in {\mathbb N}^+$, then  
			\begin{align}\label{e-f_B}
				|f_{B^*}|&\lesssim |f_{B^*}-f_{2^{l}B}|+|f_{2^{l}B}| \nonumber\\
				&\lesssim  \sum_{j=3}^{l}\fint_{2^{j}B} |f(x)-f_{2^{j}B}|\,dx+\fint_{2^{l}B}|f(y)|\,dy.
			\end{align}
\noindent Observe  that $\rho(x_B)<2^{l}r_B<2\rho(x_B)<\sigma$. Applying  \eqref{g44}, Proposition \ref{cor:equivalent character} (or Lemma \ref{12}), and H\"older's inequality to \eqref{e-f_B} yields
\begin{align}\label{important estimate}
 |f_{B^*}|\lesssim \log\left(\frac{\rho(x_B)}{r_B}\right) \varepsilon.
\end{align}
This, together with Lemma \ref{l6},  implies that for $x\in B$,
            \begin{align}\label{I3-2}
				I_3(x)\lesssim \log\left(\frac{\rho(x_B)}{r_B}\right)\left(\cfrac{r_B}{\rho(x_B)}\right)^\eta\varepsilon
				\lesssim  \varepsilon.
			\end{align}

Combing \eqref{I3-1} and \eqref{I3-2}, we obtain that if $r_B<\tilde{\sigma}_3$,
\begin{align}\label{I3-estimate}
 \left(\fint_B|I_3(x)|^2\,dx\right)^{1/2}\lesssim \varepsilon.
\end{align}
Let $\tilde{\sigma}=\min\{\tilde{\sigma}_1,\tilde{\sigma}_2,\tilde{\sigma}_3\}$. Combing \eqref{s-function}, \eqref{I1-estimate}, \eqref{I2-estimate} and \eqref{I3-estimate}, we have shown 
$$
\sup\limits_{B: \,r_B\leq \tilde{\sigma},\  r_B<\rho(x_B)}\left(\fint_{B}\left|s_L(f)(x)-C(B)\right| ^2\,dx\right)^{\frac{1}{2}}\lesssim \varepsilon,
$$
which is exactly \eqref{r<a and r<rho}.

\bigskip

  {\bf Next, let us prove \eqref{r<a and r>rho}.} Its proof is much simpler than that of \eqref{r<a and r<rho}.  
  In this case, {\it we assume that 
     $r_B\geq \rho(x_B)$}. 
     
     One may decompose $f$ as follows.             \begin{equation}\label{ee-f-decomposition}
f=f\chi_{B^*}+f\chi_{(B^*)^c}.
			\end{equation}          
            Then, by the $L^2$ boundedness of $s_L$, we have 
			\begin{equation}\label{s2}
				\begin{aligned}
					\cfrac{1}{|B|}\int_{B}|s_L(f\chi_{B^*})(x)|^2 dx
					\lesssim\cfrac{1}{|B^*|}\int_{B^*}|f(x)|^2 dx.
				\end{aligned}
			\end{equation}
           
            Denote $\tilde{\tilde{\sigma}}_1={\sigma}/4$, where $\sigma$ is as in Lemma \ref{12}.   We apply  Lemma \ref{12} to obtain that  whenever $r_B\leq \tilde{\tilde{\sigma}}_1$, there holds           
           \begin{align}\label{ee-g-smallball-21}
							\left(\cfrac{1}{|B|}\int_{B}|s_L(f\chi_{B^*})(x)|^2\, dx\right)^{1/2}\lesssim \varepsilon.
			\end{align}

            For any $x\in B$, we have
			\begin{equation}\label{e11}
				\begin{aligned}
					s_L(f\chi_{(B^*)^c})(x)
					\lesssim \left(\int_{0}^{r_B^2}\left|\int_{(B^*)^c} Q_s^L(x,y)f(y)\,dy\right|^2\frac{ds}{s}\right)^{1/2}
+\left(\int^{\infty}_{r_B^2}\left|\int_{(B^*)^c} Q_s^L(x,y)f(y)\,dy\right|^2\frac{ds}{s}\right)^{1/2}
				\end{aligned}
			\end{equation}
			
			\noindent By (i) of Proposition \ref{p1} and the fact $|y-x_B|\approx |y-x|$ when $y\in (B^*)^c$ and $x\in B$, we have 
			\begin{align}\label{ee3.1-1}
				\left(\int_{0}^{r_B^2}\left|\int_{(B^*)^c}Q_s^L(x,y)f(y)\,dy\right|^2\frac{ds}{s}\right)^{1/2} &\lesssim \left[\int_{0}^{r_B^2}\left(\int_{(B^*)^c}s^{-\frac{n}{2}}e^{-A\frac{|x_B-y|^2}{s}}|f(y)|\,dy\right)^2\frac{ds}{s}\right]^{1/2}\nonumber\\
				&\lesssim 				\left[\int_{0}^{r_B^2}\left(\frac{\sqrt{s}}{r_B}\right)^2 \frac{ds}{s}\right]^{1/2}\,\left(\int_{(B^{*})^{c}} \frac{r_B}{|x_B-y|^{n+1}} |f(y)|\,dy\right)\\
                &\lesssim \int_{(B^{*})^{c}} \frac{r_B}{|x_B-y|^{n+1}} |f(y)|\,dy.\nonumber
			\end{align}

          By Proposition \ref{p2}, we have  $\rho(x)\lesssim \rho(x_B)+|x-x_B|\lesssim r_B$ for $x\in B$.  We apply  (i) of Proposition \ref{p1}  again  to  get 
			\begin{align}\label{ee-3-r big}
\left|\int_{(B^*)^c} Q_s^L(x,y)f(y)\,dy\right|&\lesssim  \left(1+\frac{\sqrt{s}}{r_B}\right)^{-2}\,\int_{(B^{*})^{c}} \frac{\sqrt{s}}{|x_B-y|^{n+1}}|f(y)|\,dy\nonumber\\
&\lesssim \left(\frac{r_B}{\sqrt{s}}\right)\int_{(B^{*})^{c}} \frac{r_B}{|x_B-y|^{n+1}}|f(y)|\,dy
            \end{align}
 Thus,              
\begin{align}\label{ee3.1-2}
				\left(\int^{\infty}_{r_B^2}\left|\int_{(B^*)^c} Q_s^L(x,y)f(y)\,dy\right|^2\frac{ds}{s}\right)^{1/2}
				&\lesssim 
				\left[\int^{\infty}_{r_B^2} \left(\frac{r_B}{\sqrt{s}}\right)^2\frac{ds}{s}\right]^{1/2} \,\left(\int_{(B^{*})^{c}} \frac{r_B}{|x_B-y|^{n+1}}|f(y)|\,dy\right)\nonumber\\
                &\lesssim \int_{(B^{*})^{c}} \frac{r_B}{|x_B-y|^{n+1}}|f(y)|\,dy,
			\end{align}
            which, together with \eqref{ee3.1-1}, gives 
 \begin{align}\label{ee-g21}
            \left(\frac{1}{|B|}\int_B |s_L(f\chi_{(B^*)^c})(x)|^2 \,dx \right)^{1/2}\lesssim \int_{(B^{*})^{c}} \frac{r_B}{|x_B-y|^{n+1}}|f(y)|\,dy.
            \end{align}
            
            Let $\tilde{\tilde{\sigma_2}}=\sigma_1$, where $\sigma_1$ is as in Lemma \ref{l4}. We apply \eqref{ee-g21} and Lemma \ref{l4} to  obtain that for $r_B<\tilde{\tilde{\sigma_2}}$, it holds
            \begin{align}\label{ee-g-smallball-22}
            \left(\frac{1}{|B|}\int_B |s_L(f\chi_{(B^*)^c})(x)|^2 \,dx \right)^{1/2}\lesssim  \varepsilon.
            \end{align}
			
           Let $\tilde{\sigma}=\min\{\tilde{\tilde{\sigma_1}},\tilde{\tilde{\sigma_2}}\}$. Then  \eqref{r<a and r>rho} is a consequence of  \eqref{ee-g-smallball-21} and \eqref{ee-g-smallball-22}.

            \smallskip
            \bigskip
			
			{\bf Proof of  \eqref{f43}.}  The situation is very similar to that of \eqref{r<a and r>rho}.  By combining \eqref{ee-f-decomposition}, \eqref{s2}, \eqref{ee-g21}, \eqref{g88} and Lemma \ref{l4}, and setting $\tilde{M}=\max\{M, M_1\}$, (with   $M$ from  \eqref{g88} and  $M_1$ from Lemma \ref{l4}), we obtain 
			\begin{align*}
				\sup\limits_{B:\,|x_B|\geq \tilde{M},r_B\geq \rho(x_B) }\left(\fint_{B}|s_L(f)(x)|^2dx\right)^{\frac{1}{2}}&\lesssim \varepsilon, 
			\end{align*}
			which is just (\ref{f43}).

\bigskip
            
			\textbf{Proof of $(\ref{f42})$.}   The basic fact
            
			\[
			\left(\fint_{B}|g(y)-g_B|^2\,dy\right)^{1/2} \leq 2 \left(\fint_{B}|g(y)|^2\,dy,\right)^{1/2},
			\]
           together with \eqref{f43}, gives 
            $$
\sup\limits_{B:\,|x_B|\geq \tilde{M}, r_B\geq \rho(x_B)} \left(\fint_{B}\left|s_L(f)(x)-(s_L(f))_B\right|^2\,dx\right)^{1/2}\lesssim\varepsilon.	
            $$          
Thus, to prove \eqref{f42}, it suffices to prove  there exists a constant $\tilde{\tilde{M}}>0$ such that
  \begin{align}\label{ee-farball-smallr}
\sup\limits_{B: \,|x_B|\geq \tilde{\tilde{M}},\  r_B<\rho(x_B)}\left(\fint_{B}\left|s_L(f)(x)-C(B)\right| ^2\,dx\right)^{\frac{1}{2}}\lesssim \varepsilon,
            \end{align}
where $C(B)$ is as in \eqref{s:C(B)}.

The proof of \eqref{ee-farball-smallr} is  very similar to that of  \eqref{r<a and r<rho}.  By combining  \eqref{s-function}, \eqref{ee-I1}, \eqref{ee-I2}, and Lemma \ref{l6}, we have that for $r_B<\rho(x_B)$, 
\begin{align*}
&\left(\fint_{B}\left|s_L(f)(x)-C(B)\right| ^2\,dx\right)^{1/2}\\
&\lesssim \left(\fint_{B^*}|f(x)-f_{B^*}|^2 dx\right)^{1/2}+\int_{(B^*)^c} \frac{{r_B}^{\alpha_1}}{|x_B-y|^{n+\alpha_1}}|f(y)-f_{B^*}|\,dy+|f_{B^*}|\left(\frac{r_B}{\rho(x_B)}\right)^\eta.
\end{align*}
Let $\tilde{\tilde{M}}=\max\{M,M_1\}$, where  $M$ is as  in Proposition \ref{cor:equivalent character} and $M_1$ is as  in Lemma \ref{l4}. When $|x_B|\geq \tilde{\tilde{M}}$,  it follows from \eqref{g66} and Lemma \ref{l4} that
 $$
 \left(\fint_{B^*}|f(x)-f_{B^*}|^2 dx\right)^{1/2}+\int_{(B^*)^c} \frac{{r_B}^{\alpha_1}}{|x_B-y|^{n+\alpha_1}}|f(y)-f_{B^*}|\,dy\lesssim \varepsilon.
 $$

\noindent Assume that $2^{l-1}r_B\leq  \rho(x_B)<2^{l}r_B,$  for some $l\in {\mathbb N}^+$.  By \eqref{e-f_B}, \eqref{g66} and \eqref{g88}, we have that when $|x_B|\geq \tilde{\tilde{M}}$, there holds
    \begin{align*}
				|f_{B^*}|\left(\frac{r_B}{\rho(x_B)}\right)^\eta
				&\lesssim  \left(\sum_{j=2}^{l}\fint_{2^{j}B} |f(x)-f_{2^{j}B}|\,dx+\fint_{2^{l}B}|f(y)|\,dy\right)\left(\frac{r_B}{\rho(x_B)}\right)^\eta\\
                &\lesssim \left(\frac{r_B}{\rho(x_B)}\right)^\eta\log\left(\cfrac{\rho(x_B)}{r_B}\right)\varepsilon\\
                &\lesssim \varepsilon.
			\end{align*}		
     
          This completes the proof of \eqref{ee-farball-smallr}.

\bigskip

\section{The mapping $S_L:\ {\rm CMO}_L\to {\rm CMO}_L$}

Since the proof strategy for $S_L$ parallels  that for $s_L$ in Section 3, we only give a sketch below, highlighting the main differences.          
 
Assume that $f\in {\rm CMO}_L(\Rn)$. It follows \cite[Theorem 6]{LL2011AdvMath} that   $S_L(f)\in {\rm BMO}_L$ and   $\|S_L(f)\|_{{\rm BMO}_L}\lesssim\|f\|_{{\rm BMO}_L}$. Thus, to prove that $S_L(f)\in {\rm CMO}_L$,  by Proposition \ref{CMOL-new characterization},  it suffices to show that for any $\varepsilon>0$, there exist positive constants $\tilde{\sigma}\ll1$ and $\tilde{M}\gg1$ such that  
		\begin{align*}
	&\sup_{B:r_{B}\leq\tilde{\delta}}\left(\fint_{B}|S_L(f)(x)-S_L(f)_{B}|^{2}\,dx\right)^{1/2}+
				\sup_{B:|x_B|\geq \tilde{M}}\left(\fint_{B}|S_L(f)(x)-S_L(f)_{B}|^{2}\,dx\right)^{1/2}\\
                &\hskip 1cm+
				\sup_{B:|x_B|\geq \tilde{M},\,r_{B}\geq\rho(x_{B})}\left(\fint_{B}|S_L(f)(x)|^{2}\,dx\right)^{1/2}\lesssim\varepsilon.
			\end{align*}

Denote  $B^*:=B(x_B,4r_B)$. By a similar argument in Section 3, it suffices to show that the following inequalities hold.  
		
 (i) If $r_B<\rho(x_B)$, then  there exist  constants  $\tilde{C}(B)$ (depending only on  $B$)  and a positive number $\eta$ such that
\begin{align*}
\left(\fint_{B}\left|S_L(f)(x)-\tilde{C}(B)\right| ^2\,dx\right)^{1/2}
&\lesssim \left(\cfrac{1}{|B^*|}\int_{B^*}|f(x)-f_{B^*}|^2 dx\right)^{1/2}+\big|f_{B^*}\big|\left(\frac{r_B}{\rho(x_B)}\right)^\eta\\
&\hskip 0.5cm+\int_{(B^*)^c} \frac{{r_B}^{1/4}}{|x_B-z|^{n+1/4}}|f(z)-f_{B^*}|\,dz.
\end{align*}

(ii) If $r_B\geq\rho(x_B)$,  then
\begin{align*}
\left(\fint_{B}\left|S_L(f)(x)\right|^2\,dx\right)^{1/2}\lesssim \left(\cfrac{1}{|B^*|}\int_{B^*}|f(x)|^2 dx\right)^{1/2}+\int_{(B^{*})^{c}} \frac{r_B}{|x_B-y|^{n+1}}|f(y)|\,dy.
\end{align*}

   \bigskip

{\bf Proof of (i)} \  Assume that     $r_B<\rho(x_B)$.
We decompose $f$ as the sum $f = f_1 + f_2 + f_3$, 
where each  $f_i$ is given by \eqref{el}.	
Define $\tilde{C}(B)$  as follows:
$$
\tilde{C}(B):=\left\|Q^L_sf_2(y)\chi_{[r_B^2,\,\infty)}(s)\chi_{B(x_B,\sqrt{s})}(y)+Q^L_sf_3(y)\chi_{[r_B\rho(x_B),\,\infty)}(s)\chi_{B(x_B,\sqrt{s})}(y)
\right\|_{L^2({\mathbb R}^{n+1}_{+},s^{-n/2-1}dyds)}.
$$

We claim  that $\tilde{C}(B)\leq \tilde{C}_1(B)+\tilde{C}_2(B)<\infty$, where
$$
\tilde{C}_1(B):=\left(\int_{r_B^2}^{\infty}\int_{B(x_B,\sqrt{s})}|Q_s^Lf_2(y)|^2\,\frac{dyds}{s^{\frac{n}{2}+1}}\right)^{\frac{1}{2}} \quad {\rm and} \quad
\tilde{C}_2(B):=\left(\int_{r_B\rho(x_B)}^{\infty}\int_{B(x_B,\sqrt{s})}|Q_s^Lf_3(y)|^2\,\frac{dyds}{s^{\frac{n}{2}+1}}\right)^{\frac{1}{2}}.
$$
In fact, by  (ii) of Proposition \ref{p1}, if $|y-x_B|<\sqrt{s}$, then 
        \begin{align*}
            \left|Q_s^L(y,z)-Q_s^L(x_B,z)\right|&\lesssim s^{-\frac{n}{2}}e^{-A\frac{|x_B-z|^2}{s}}\left(1+\cfrac{\sqrt{s}}{\rho(x_B)}\right)^{-2},
        \end{align*}
 which, together with (i) of Proposition \ref{p1}, yields that       
\begin{align}\label{S:QL(y,z)}
				\left|Q_s^L(y,z)\right|\leq \left|Q_s^L(y,z)-Q_s^L(x_B,z)\right|+\left|Q_s^L(x_B,z)\right|\lesssim s^{-\frac{n}{2}}e^{-A\frac{|x_B-z|^2}{s}}\left(1+\cfrac{\sqrt{s}}{\rho(x_B)}\right)^{-2}.
\end{align}
   Then we apply \eqref{S:QL(y,z)} to obtain
  \begin{align*}
      \tilde{C}_1(B)^2 &\lesssim \int_{r_B^2}^\infty\int_{|y-x_B|<\sqrt{s}}\left[ \int_{\Rn} s^{-\frac{n}{2}}e^{-A\frac{|z-x_B|^2}{s}}\left(1+\cfrac{\sqrt{s}}{\rho(x_B)}\right)^{-2}|f_2(z)|\,dz \right]^2\cfrac{dyds}{s^{\frac{n}{2}+1}}\\
      &\lesssim \int_{r_B^2}^\infty\left[ \int_{\Rn} s^{-\frac{n}{2}}e^{-A\frac{|z-x_B|^2}{s}}\left(1+\cfrac{\sqrt{s}}{\rho(x_B)}\right)^{-2}|f_2(z)|\,dz \right]^2\cfrac{ds}{s}.   
  \end{align*}
 By \eqref{s:C_1(B)}, we have $ \tilde{C}_1(B)<\infty$.   Similarly, by \eqref{S:QL(y,z)} and  \eqref{s:C_2(B)}, we have $\tilde{C}_2(B)<\infty$.
  
For any $x\in B$,  we apply Minkowski's inequality to obtain
\begin{align*}
    |S_L(f)(x)-\tilde{C}(B)|&\leq\left\| Q_s^Lf_1(y)\chi_{B(x,\sqrt{s})}(y)\right\|_{L^2(\IR^n_+,\,s^{-\frac{n}{2}-1}dyds)}\\
    &\quad +\left\|Q_s^Lf_2(y)\chi_{B(x,\sqrt{s})}(y)- Q^L_sf_2(y)\chi_{[r_B^2,\infty)}(s)\chi_{B(x_B,\sqrt{s})}(y)\right\|_{L^2(\IR^n_+,\,s^{-\frac{n}{2}-1}dyds)}\\
&\quad +\left\|Q_s^Lf_3(y)\chi_{B(x,\sqrt{s})}(y)- Q^L_sf_3(y)\chi_{[r_B\rho(x_B),\infty)}(s)\chi_{B(x_B,\sqrt{s})}(y)\right\|_{L^2(\IR^n_+,\,s^{-\frac{n}{2}-1}dyds)}\\
    &=: \tilde{I}_1(x)+\tilde{I}_2(x)+\tilde{I}_3(x).
\end{align*}

By the $L^2$ boundedness of $S_L$, we  obtain
				\begin{align}\label{S:ee-I1}
	\fint_{B}|{\tilde I}_1(x)|^2 dx=\fint_{B}|S_L(f_1)(x)|^2 dx\lesssim {|B|}^{-1}\int_{\mathbb{R}^n}|f_1(x)|^2 dx
				\lesssim\cfrac{1}{|B^*|}\int_{B^*}|f(x)-f_{B^*}|^2 dx.	
				\end{align}

Given sets $E$ and $F$,  write  $E\Delta F:=(E\backslash F) \cup (F\backslash E)$.   For $x\in B$, we  estimate ${\tilde  I}_2(x)$ as follows.  
			\begin{align*}
				{\tilde  I}_2(x)
                &\leq \left(\int^{r_B^2}_0\int_{B(x,\sqrt{s})}\left|Q_s^Lf_2(y)\right|^2\cfrac{dyds}{s^{\frac{n}{2}+1}}\right)^\frac{1}{2}+\left(\int_{r_B^2}^\infty\int_{B(x,\sqrt{s})\Delta B(x_B,\sqrt{s})}\left|Q_s^Lf_2(y)\right|^2\cfrac{dyds}{s^{\frac{n}{2}+1}}\right)^\frac{1}{2}\\
		&=:\tilde{I}_{21}(x)+\tilde{I}_{22}(x).
			\end{align*}
		
        For any $x\in B, z\in (B^{*})^c, y\in B(x,\sqrt{s})$ with $ \sqrt{s}<r_B$, it holds that $|z-y|\geq |z-x_B|/2$.
			Therefore, 
			\begin{align}\label{S:I_21}
				({\tilde I}_{21}(x))^2&\lesssim\int^{r_B^2}_0\int_{|y-x|<\sqrt{s}}\left( \int_{(B^{*})^c}s^{-\frac{n}{2}}e^{-A\frac{|z-y|^2}{s}}|f_2(z)| \,dz\right)^2\cfrac{dyds}{s^{\frac{n}{2}+1}}\nonumber\\
    &\lesssim\int^{r_B^2}_0\int_{|y-x|<\sqrt{s}}\left( \int_{(B^{*})^c}s^{-\frac{n}{2}}e^{-A\frac{|z-x_B|^2}{4s}}|f_2(z)| \,dz\right)^2\cfrac{dyds}{s^{\frac{n}{2}+1}}\nonumber\\
				&\lesssim  \int^{r_B^2}_0\left( \int_{(B^{*})^c}s^{-\frac{n}{2}}e^{-cA\frac{|z-x_B|^2}{4s}}|f_2(z)|\,dz \right)^2\cfrac{ds}{s}\\
                & \lesssim
   \left(\int_{(B^*)^c}\cfrac{r_B}{|z-x_B|^{n+1}} |f(z)-f_{B^*}|\,dz\right)^2,\nonumber
			\end{align}   
		where in the last inequality we have used \eqref{I21}.	
						
		Observe   the following geometric fact:  
\begin{align}\label{geometric fact}
	\left|B(x,\sqrt{s})\Delta B(x_B,\sqrt{s}) \right|\lesssim r_Bs^{\frac{n-1}{2}},\quad \text{for} \ \ |x-x_B|<r_B\leq \sqrt{s},
			\end{align}
which, together with \eqref{S:QL(y,z)}, implies that	\begin{align}\label{S:I_22}
				({\tilde I}_{22}(x))^2&\lesssim  \int_{r_B^2}^\infty \int_{B(x,\sqrt{s})\Delta B(x_B,\sqrt{s})} \left(\int_{(B^{*})^c} \frac{(\sqrt{s})^{1/4}}{(\sqrt{s}+|z-y|)^{n+\frac{1}{4}}} |f_2(z)|\,dz \right)^2\cfrac{dyds}{s^{\frac{n}{2}+1}}\nonumber\\
				&\lesssim  \int_{r_B^2}^\infty \left(\frac{r_B}{\sqrt{s}}\right) \left(\int_{(B^{*})^c} \frac{(\sqrt{s})^{1/4}}{(\sqrt{s}+|z-x_B|)^{n+\frac{1}{4}}}|f_2(z)|\,dz\right)^2\frac{ds}{s}\\
                &\lesssim \left(\int_{(B^*)^c} \frac{{r_B}^{1/4}}{|z-x_B|^{n+\frac{1}{4}}}|f(z)-f_{B^*}|\,dz\right)^2,\nonumber
			\end{align}
            where in the second inequality we  used the fact $\sqrt{s}+|z-x_B|\leq 3(\sqrt{s}+|z-y|)$.

\smallskip
                        
			We now  estimate $\tilde{I}_3(x)$ for $x\in B$.
		Since $x\in B$, we  have 		\begin{align*}
					\tilde{I}_3(x)&\lesssim |f_{B^*}|\left(\int^{r_B\rho(x_B)}_0\int_{B(x,\sqrt{s})}\left|Q_s^L(1)(y)\right|^2\cfrac{dyds}{s^{\frac{n}{2}+1}}\right)^{1/2}\\
&\quad+|f_{B^*}|\left(\int_{r_B\rho(x_B)}^\infty\int_{B(x,\sqrt{s})\Delta B(x_B,\sqrt{s})}\left|Q_s^L(1)(y)\right|^2\cfrac{dyds}{s^{\frac{n}{2}+1}}\right)^{1/2}\\
					&=: |f_{B^*}|\left(\tilde{I}_{31}(x)+\tilde{I}_{32}(x)\right).
				\end{align*}
                
                Note that if $y\in B(x,\sqrt{s})$, then  the facts $x\in B$  and $s<r_B\rho(x_B)<\rho(x_B)^2$ will imply that 
                $$
                |y-x_B|\leq |y-x|+|x-x_B|\leq \sqrt{s}+r_B\leq 2\rho(x_B),
                $$
                which, by Proposition \ref{p2}, yields  $\rho(y)\approx \rho(x_B)$.

Then, we apply (iii) of Proposition \ref{p1} to obtain  
				\begin{align*}
	(\tilde{I}_{31}(x))^2&\lesssim 
					\int^{r_B\rho(x_B)}_0\int_{B(x,\sqrt{s})}\left(\cfrac{\sqrt{s}}{\rho(y)}\right)^{2\delta}\cfrac{dyds}{s^{\frac{n}{2}+1}}       \lesssim 
					\int^{r_B\rho(x_B)}_0\left(\cfrac{\sqrt{s}}{\rho(x_B)}\right)^{2\delta}\cfrac{ds}{s}\lesssim \left(\cfrac{r_B}{\rho(x_B)}\right)^{\delta}.
				\end{align*}
				It follows from (i) of Proposition  \ref{p1} that $|Q_s^L(1)(y)|\lesssim 1$. This,  combined with \eqref{geometric fact}, implies that 
						\begin{align*}
					(\tilde{I}_{32}(x))^2&\lesssim \int_{r_B\rho(x_B)}^\infty|B(x,\sqrt{s})\Delta B(x_B,\sqrt{s})|\frac{ds}{s^{\frac{n}{2}+1}}
					\lesssim \int_{r_B\rho(x_B)}^\infty \cfrac{r_B}{\sqrt{s}}\,\cfrac{ds}{s}
					\lesssim \left(\cfrac{r_B}{\rho(x_B)}\right)^{1/2}.
				\end{align*}
Letting $\eta=\min\{\delta/2, 1/4\}$,  we have proved
\begin{align}\label{S:I_3}
\tilde{I}_3(x)\lesssim |f_{B^*}|\left(\frac{r_B}{\rho(x_B)}\right)^\eta.    
\end{align}
By combining \eqref{S:ee-I1}, \eqref{S:I_21},\eqref{S:I_22} and \eqref{S:I_3}, we have shown (i).

\bigskip

{\bf Proof of  (ii).} Assume that 
     $r_B\geq \rho(x_B)$. 
We decompose $
f=f\chi_{B^*}+f\chi_{(B^*)^c}$.  

It follows from  the $L^2$ boundedness of $S_L$ that
\begin{align}\label{S:ee4.0}
\left(\fint_{B}\left|S_L(f\chi_{B^*})(x)\right|^2\,dx\right)^{1/2}
\lesssim \left(\cfrac{1}{|B^*|}\int_{B^*}|f(z)|^2 dz\right)^{1/2}.
\end{align}

Next, we estimate  $S_L(f\chi_{(B^*)^c})(x)$ for $x\in B$.
Since $x,x_B\in B$ and $r_B\geq\rho(x_B)$,  Proposition \ref{p2} yields 
$$\rho(x)\lesssim \rho(x_B)+|x-x_B|\lesssim r_B.
$$   
Then, for $|y-x|<\sqrt{s}$, we have
\begin{align}\label{S:ee4.1}
    \left|Q_s^L(y,z)\right|&\leq \left|Q_s^L(y,z)-Q_s^L(x,z)\right|+\left|Q_s^L(x,z)\right|\nonumber\\
    &\lesssim s^{-\frac{n}{2}}e^{-A\frac{|x-z|^2}{s}}\left(1+\cfrac{\sqrt{s}}{\rho(x)}\right)^{-2}\lesssim s^{-\frac{n}{2}}e^{-A\frac{|x_B-z|^2}{4s}}\left(1+\cfrac{\sqrt{s}}{r_B}\right)^{-2},
\end{align}
where  the second inequality follows from (i) and (ii) of Proposition \ref{p1}.

By \eqref{S:ee4.1}, we can obtain that 
for $x\in B$
\begin{align*}
    S_L(f\chi_{(B^*)^c})(x) &\leq \left(\int_0^{\infty} \int_{B(x,\sqrt{s})}|Q_s^L(f\chi_{(B^*)^c})(y)|^2\, \frac{dyds}{s^{n/2+1}}\right)^{1/2}\\
    &\lesssim \left(\int_0^{\infty} \int_{B(x,\sqrt{s})}\left(\int_{(B^*)^c}s^{-\frac{n}{2}}e^{-A\frac{|x_B-z|^2}{4s}}\left(1+\cfrac{\sqrt{s}}{r_B}\right)^{-2} |f(z)|\,dz\right)^2\, \frac{dyds}{s^{n/2+1}}\right)^{1/2}\\
    &\lesssim \left(\int_0^{\infty} \left(\int_{(B^*)^c}s^{-\frac{n}{2}}e^{-A\frac{|x_B-z|^2}{4s}}\left(1+\cfrac{\sqrt{s}}{r_B}\right)^{-2} |f(z)|\,dz\right)^2\, \frac{ds}{s}\right)^{1/2},
\end{align*}
which, via a similar argument to \eqref{ee-g21}, implies that    
\begin{align}\label{S:ee4.2}
S_L(f\chi_{(B^*)^c})(x) 
\lesssim \int_{(B^{*})^{c}} \frac{r_B}{|x-y|^{n+1}}|f(y)|\,dy.
\end{align}

Combining \eqref{S:ee4.0} and \eqref{S:ee4.2}, we complete the proof of (ii).

\bigskip

\section{The mapping $T_L^*:\ {\rm CMO}_L \to {\rm CMO}_L$  }

 The proof framework   for $T_L^*$   parallels that  established  for  $s_L$ in Section 3, but two key differences arise. First, the $L^2$ norm $\left(\int_{0}^\infty |tLe^{-tL}(f)|^2 \frac{dt}{t}\right)^{1/2}$ is replaced by the supremum  $\sup_{t>0}|e^{-tL}(f)|$. Second, the function  $e^{-tL}(1)$ fails to  satisfy  (iii) of Proposition \ref{p1}.  In what follows, we outline the main points that require modification.

From \cite{LL2011AdvMath},  we have the boundedness $T^*_L:  {\rm BMO}_L \to {\rm BMO}_L$ with the estimate  $\|T^*_L(f)\|_{{\rm BMO}_L}\lesssim\|f\|_{{\rm BMO}_L}$. Consequently, for $f\in {\rm CMO}_L$,  to establish $T_L^*(f)\in {\rm CMO}_L$, it is enough, by Proposition \ref{CMOL-new characterization},  to prove that for any $\varepsilon>0$, there exist positive constants ${\sigma}'\ll 1$ and ${M}'\gg1$ satisfying  
\begin{align*}
	&\sup_{B:r_{B}\leq {\delta}'}\left(\fint_{B}|T^*_L(f)(x)-\left(T^*_L(f)\right)_{B}|^{2}\,dx\right)^{1/2}+
				\sup_{B:|x_B|\geq {M}'}\left(\fint_{B}|T^*_L(f)(x)-\left(T^*_L(f)\right)_{B}|^{2}\,dx\right)^{1/2}\\
                &\hskip 1cm+
				\sup_{B:|x_B|\geq {M}',\,r_{B}\geq\rho(x_{B})}\left(\fint_{B}|T^*_L(f)(x)|^{2}\,dx\right)^{1/2}\lesssim\varepsilon.
			\end{align*}

Let $B^*:=B(x_B,4r_B)$. By the same reasoning as in Section 3, it is enough to establish  the following two estimates.
		
 (i) If $r_B<\rho(x_B)$, then  there exist  constants $\alpha>0$, $\eta>0$, and $\hat{C}(B)$ (depending only on  $B$),  such that
\begin{align*}
\left(\fint_{B}\left|T^*_L(f)(x)-\hat{C}(B)\right| ^2\,dx\right)^{1/2}
&\lesssim \left(\cfrac{1}{|B^*|}\int_{B^*}|f(x)-f_{B^*}|^2 dx\right)^{1/2}+\big|f_{B^*}\big|\left(\frac{r_B}{\rho(x_B)}\right)^\eta\\
&\hskip 0.5cm+\int_{(B^*)^c} \frac{{r_B}^{\alpha}}{|x_B-z|^{n+\alpha}}|f(z)-f_{B^*}|\,dz.
\end{align*}

(ii) If $r_B\geq\rho(x_B)$,  then
\begin{align*}
\left(\fint_{B}\left|T^*_L(f)(x)\right|^2\,dx\right)^{1/2}\lesssim \left(\cfrac{1}{|B^*|}\int_{B^*}|f(x)|^2 dx\right)^{1/2}+\int_{(B^{*})^{c}} \frac{r_B}{|x_B-y|^{n+1}}|f(y)|\,dy.
\end{align*}

   \bigskip

{\bf Proof of (i).} \  Assume that     $r_B<\rho(x_B)$.
 We write  $f=f_1+f_2+f_3$, where each $f_i$ is given by \eqref{el}.		
Define $\hat{C}(B)$  as follows:
\begin{align*}
    \hat{C}(B):= \sup_{t>0}\left|e^{-tL}f_2(x_B)\chi_{[r_B^2,\infty)}(t)+f_3\chi_{(0,r_B\rho(x_B))}(t) +e^{-tL}(f_3)(x_B)\chi_{[r_B\rho(x_B),\infty)}(t)  \right|.
\end{align*}
It is worth noting that  $\hat{C}_B$ differs from  $C_B$ in \eqref{s:C(B)} in two ways: firstly, the $L^2$ norm is replaced by the supremum; secondly,  we introduce an additional term $f_3\chi_{(0,r_B\rho(x_B))}(t)$, which is due to the fact that $e^{-tL}(1)$ fails to  satisfy  (iii) of Proposition \ref{p1}.  

It is straightforward to check that $\hat{C}(B)<\infty$. On the one hand, from the proof of \eqref{s:C_1(B)} we get  
$$\sup_{t\geq r_B^2}\left|e^{-tL}f_2(x_B)\right|\lesssim \sup_{t\geq r_B^2} \left(\frac{\rho(x_B)}{\sqrt{t}}\right)^2 \left(\frac{\sqrt{t}}{r_B}\right)  \|f\|_{\rm BMO}\leq \left(\frac{\rho(x_B)}{r_B}\right)^2   \|f\|_{\rm BMO}<\infty.
$$
On the other hand,  Proposition \ref{6a1} yields, 
$$\sup_{0<t<r_B\rho(x_B)}|f_3|+\sup_{t\geq r_B\rho(x_B)}\left|e^{-tL}(f_3)(x_B)\right|\lesssim |f_{B^*}|<\infty.
$$

For any $x\in B$,  we apply Minkowski's inequality to obtain
\begin{align*}
    \left|T^*_L(f)(x)-\hat{C}_B \right| &\leq \sup_{t>0}\left|e^{-tL}f_1(x)\right|+\sup_{t>0}\left|e^{-tL}f_2(x)-e^{-tL}f_2(x_B)\chi_{[r_B^2,\infty]}\right|\\
    &\quad+\sup_{t>0}\left|e^{-tL}f_3(x)-f_3\chi_{(0,r_B\rho(x_B))}(t)-e^{-tL}(f_3)(x_B)\chi_{[r_B\rho(x_B),\infty)}(t)\right|\\
    &=:{J}_1(x)+{J}_2(x)+J_3(x).
\end{align*}

 We note that the operator $T^*_L$ is bounded on $L^2(\Rn)$; indeed, this  follows from  the pointwise estimate  $|T^*_Lf(x)|\lesssim {\mathcal M}f(x)$,  where ${\mathcal M}$ is the Hardy--Littlewood maximal operator. Consequently,
\begin{align}\label{ee-5-J1}
    \fint_B|J_1(x)|^2\,dx&\lesssim \frac{1}{|B|}\int_{\Rn}\left|f_1(x) \right|^2\,dx\lesssim \fint_{B^*}|f(x)-f_{B^*}|^2\,dx.
\end{align}

Consider $J_2(x)$.  One has
\begin{align*}
    J_2(x)&\leq \sup_{0<t<r_B^2}\left|e^{-tL}(f_2)(x)(t)\right|+\sup_{t\geq r_B^2}\left|e^{-tL}(f_2)(x)-e^{-tL}(f_2)(x_B)\right|\\
    &=:J_{21}(x)+J_{22}(x). 
\end{align*}

Observe that  $|y - x_B| \approx |x - y|$ for $x, x_B\in B$ and $y \in (B^*)^c$. This, together with  (i) from Proposition \ref{6a1}, implies
\begin{align}\label{ee-5-J21}
    J_{21}(x)&\lesssim \sup_{0<t<r_B^2}\int_{(B^*)^c}\frac{{\sqrt{t}}}{(\sqrt{t}+|x-y|)^{n+1}}|f(y)-f_{B^*}|\,dy\nonumber\\
    &\lesssim \sup_{0<t<r_B^2}\int_{(B^*)^c}\frac{\sqrt{t}}{|x_B-y|^{n+1}}|f(y)-f_{B^*}|\,dy\\
    &\lesssim \int_{(B^*)^c}\frac{r_B}{|x_B-y|^{n+1}}|f(y)-f_{B^*}|\,dy.\nonumber
\end{align}
For $t\geq r_B^2$,  since  $|x-x_B|<r_B\leq \sqrt{t}$, applying Proposition \ref{6a2} yields
\begin{align}\label{ee-5-J22}
    J_{22}(x)&\lesssim \sup_{r_B^2<t<\infty}\left(\frac{r_B}{\sqrt{t}}\right)^{\delta}\int_{(B^*)^c}\frac{(\sqrt{t})^\alpha}{(\sqrt{t}+|x_B-y|)^{n+\alpha}}|f(y)-f_{B^*}|\,dy\nonumber\\
    &\lesssim\sup_{r_B^2<t<\infty}\left(\frac{r_B}{\sqrt{t}}\right)^{\delta-\alpha}\int_{(B^*)^c}\frac{(r_B)^\alpha}{|x_B-y|^{n+\alpha}}|f(y)-f_{B^*}|\,dy\\
    &\lesssim \int_{(B^*)^c}\frac{(r_B)^\alpha}{|x_B-y|^{n+\alpha}}|f(y)-f_{B^*}|\,dy,    \nonumber
\end{align}
where we  choose $\alpha<\delta$.

Consider $J_3(x)$.  We have 
\begin{align*}
    J_3(x)&\leq \sup_{0<t<r_B\rho(x_B)}\left|e^{-tL}(f_{B^*})(x)-f_{B^*}\right|+\sup_{t\geq r_B\rho(x_B)}\left|e^{-tL}(f_B^*)(x)-e^{-tL}(f_{B^*})(x_B)\right|\\
    &=: J_{31}+J_{32}.
\end{align*}

\noindent For $J_{31}(x)$, we apply the fact that $e^{-t(-\Delta)}(1)=1$ and  Proposition \ref{6a3} to obtain 
\begin{align}\label{ee-5-J31}
    J_{31}&=\sup_{0<t<r_B\rho(x_B)}\left| e^{-tL}(f_{B^*})(x)-e^{-t(-\Delta)}(f_{B^*})(x) \right|\nonumber\\
    &\lesssim |f_{B^*}|\sup_{0<t<r_B\rho(x_B)} \left(\frac{\sqrt{t}}{\rho(x_B)}\right)^\delta\int_{\IR^n}\cfrac{{\sqrt{t}}}{(\sqrt{t}+|x-y|)^{n+1}}\,dy\\
    &\lesssim |f_{B^*}| \left(\frac{r_B}{\rho(x_B)}\right)^{\delta/2},\nonumber
\end{align}
where in the first inequality we used $\rho(x)\approx\rho(x_B)$ since $|x-x_B|<\rho(x_B)$.


Consider $J_{32}(x)$, for $x\in B$.  Since $|x-x_B|\leq r_B\leq \sqrt{t}$, we apply Proposition \ref{6a2} to obtain
\begin{align}\label{ee-5-J32}
    J_{32}(x)
    &\lesssim |f_{B^*}|\sup_{t\geq r_B\rho(x_B)}\left(\frac{|x-x_B|}{\sqrt{t}}\right)^{\delta}\int_{\IR^n}\frac{\sqrt{t}}{(\sqrt{t}+|x-y|)^{n+1}}\,dy\nonumber\\
    &\lesssim |f_{B^*}|\sup_{t\geq r_B\rho(x_B)}\left(\frac{|x-x_B|}{\sqrt{t}}\right)^{\delta}\lesssim |f_{B^*}|\left(\frac{r_B}{\rho(x_B}\right)^{\delta/2}.
\end{align}

By combining \eqref{ee-5-J1}-\eqref{ee-5-J32}, we have proved (i).

\bigskip

 {\bf Proof of (ii).} Assume that 
     $r_B\geq \rho(x_B)$. 
We decompose $
f=f\chi_{B^*}+f\chi_{(B^*)^c}$.  

It follows from  the $L^2$ boundedness of $T^*_L$ that
\begin{align}\label{S:ee5.0}
\left(\fint_{B}\left|T^*_L(f\chi_{B^*})(x)\right|^2 dx\right)^{1/2}
\lesssim \left(\cfrac{1}{|B^*|}\int_{B^*}|f(y)|^2 dy\right)^{1/2}.
\end{align}

For $x\in B$, we estimate the term  $T^*_L(f\chi_{(B^*)^c})(x)$ as

$$T^*_L(f\chi_{(B^*)^c})(x)\leq \sup_{0<t<r_B^2}|e^{-tL}(f\chi_{(B^*)^c})(x)| +\sup_{t\geq r_B^2}|e^{-tL}(f\chi_{(B^*)^c})(x)|.
$$

\noindent By Proposition \ref{6a1} and the fact that $|x_B-y|\approx |x-y|$ for $x,x_B\in B$ and $y\in (B^*)^c$, we have 
\begin{align}\label{ee-5-1}
   \sup_{0<t<r_B^2} \left|e^{-tL}(f\chi_{(B^*)^c})(x)\right|
      \lesssim \sup_{0<t<r_B^2} \int_{(B^*)^c}\frac{\sqrt{t}}{|x_B-y|^{n+1}}|f(y)|\,dy      
       \lesssim  \int_{(B^*)^c}\frac{r_B}{|x_B-y|^{n+1}}|f(y)|\,dy.
\end{align}

On the other hand, by an argument similar to that of \eqref{ee-3-r big},
	 we have
\begin{align*}
         \left|e^{-tL}(f\chi_{(B^*)^c})(x)\right|        \lesssim \left(\frac{r_B}{\sqrt{t}}\right) \int_{(B^*)^c}\frac{r_B}{|x_B-y|^{n+1}}|f(y)|\,dy,
     \end{align*}
which yields
\begin{align}\label{ee-5-2}
\sup_{t\geq r_B^2}  \left|e^{-tL}(f\chi_{(B^*)^c})(x)\right|\lesssim  \int_{(B^*)^c}\frac{r_B}{|x_B-y|^{n+1}}|f(y)|\,dy.
\end{align}
By combining \eqref{S:ee5.0}, \eqref{ee-5-1} and \eqref{ee-5-2}, we complete the proof of (ii).

\bigskip
\bigskip

\noindent {\bf Acknowledgments}   The research is supported by National Key R$\&$D Program of China\\  
2022YFA1005700.   L. Song  is supported by NNSF of China (Nos. 12471097 and 12371105).   
Q. Lin is supported by NNSF of China (12501126), Guangdong Basic and Applied Basic Research Foundation (2024A1515110227) and STU Scientific Research Initiation Grant (NTF24015T).

%
%
%
%
%
%
%

\bibliographystyle{plain}
\bibliography{ref} 

@article{BDS1981AnnMath,
	title={Weak {$L^\infty$} and {BMO}},
	author={Bennett, C. and DeVore, R.A. and Sharpley, R.},
	journal={Ann. of Math.},
	volume={113},
	number={3},
	pages={601--611},
	year={1981},
}

@article{CW1977BAMS,
	title        = {Extensions of Hardy spaces and their use in analysis},
	author       = { Coifman, R.R. and Weiss, G.},
	journal      = {Bull. Amer. Math. Soc.},
	volume       = {83},
	number       = {4},
	pages        = {569--646},
	year         = {1977},
}

@article {DDSTY2008MichiganMathMathJ,
    AUTHOR = {Deng, D.G. and Duong, X.T. and Song, L. and Tan,          C.Q. and Yan, L.X.},
     TITLE = {Functions of vanishing mean oscillation associated with
              operators and applications},
   JOURNAL = {Michigan Math. J.},
    VOLUME = {56},
      YEAR = {2008},
    NUMBER = {3},
     PAGES = {529--550},
      ISSN = {0026-2285,1945-2365},
   MRCLASS = {47D06 (42B25 46E15 47A60 47B38 47F05)},
  MRNUMBER = {2488724},
MRREVIEWER = {Galia\ D.\ Dafni},
}

@article {DGMTZ2005MathZ,
	AUTHOR = {Dziuba{\'n}ski, J. and Garrig{\'o}s, G. and Mart{\'\i}nez, T. and Torrea, J.L. and Zienkiewicz, J.},
	TITLE = {{$BMO$} spaces related to {S}chr\"odinger operators with
		potentials satisfying a reverse {H}\"older inequality},
	JOURNAL = {Math. Z.},
	FJOURNAL = {Mathematische Zeitschrift},
	VOLUME = {249},
	YEAR = {2005},
	NUMBER = {2},
	PAGES = {329--356},
	ISSN = {0025-5874,1432-1823},
	MRCLASS = {35J10 (42B30 42B35)},
	MRNUMBER = {2115447},
	MRREVIEWER = {Takayoshi\ Ogawa},
}

@article {DZ1997StusiaMath,
    AUTHOR = {Dziuba\'nski, J. and  Zienkiewicz, J.},
     TITLE = {Hardy spaces associated with some {S}chr\"odinger operators},
   JOURNAL = {Studia Math.},
    VOLUME = {126},
      YEAR = {1997},
    NUMBER = {2},
     PAGES = {149--160},
      ISSN = {0039-3223,1730-6337},
   MRCLASS = {42B30 (35J10 46E15 47D06)},
  MRNUMBER = {1472695},
MRREVIEWER = {Norman\ J.\ Weiss},
      
}

@article {DZ2003ColloqMath,
    AUTHOR = {Dziuba\'nski, J. and Zienkiewicz, J.},
     TITLE = {{$H^p$} spaces associated with {S}chr\"odinger operators with
              potentials from reverse {H}\"older classes},
   JOURNAL = {Colloq. Math.},
    VOLUME = {98},
      YEAR = {2003},
    NUMBER = {1},
     PAGES = {5--38},
      ISSN = {0010-1354,1730-6302},
   MRCLASS = {42B30 (35J10 47D06 47D08)},
  MRNUMBER = {2032068},
MRREVIEWER = {Ferenc\ Weisz},
}

@article {HLW2025NYJ,
    AUTHOR = {Han, X.T. and Li, J. and Wu, L.C.},
     TITLE = {Endpoint boundedness of singular integrals: CMO space associated to Schr\"odinger operators},
   JOURNAL = {New York J. Math.},
    VOLUME = {31},
      YEAR = {2025},
    NUMBER = {},
     PAGES = {1345--1380},
      ISSN = {1076-9803},
   MRCLASS = {42B20 (42B25 42B35)},
  MRNUMBER = {4960628},
MRREVIEWER = {},
}

@article{FS1972ActaMath,
	title        = {{$H^p$} Spaces of Several Variables},
	author       = {Fefferman, C. and Stein, E.M.},
	journal      = {Acta Math.},
 VOLUME = {129},
      YEAR = {1972},
    NUMBER = {3-4},
     PAGES = {137--193},
}

@article{Goldberg1979DukeMathJ,
	title={A local version of real {H}ardy spaces},
	author={Goldberg, D.M.},
	journal={Duke Math. J.},
	VOLUME = {46},
      YEAR = {1979},
    NUMBER = {1},
     PAGES = {27--42},
}

@article{JN1961CPAM,
	title={On functions of bounded mean oscillation},
	author={John, F. and Nirenberg, L.},
	journal={Comm. Pure Appl. Math.},
	volume={14},
	number={3},
	pages={415--426},
	year={1961},
	publisher={Wiley Online Library}
}

@article {Kurtz1987ProcAmerMathSoc,
    AUTHOR = {Kurtz, D.S.},
     TITLE = {Littlewood--{P}aley operators on {BMO}},
   JOURNAL = {Proc. Amer. Math. Soc.},
    VOLUME = {99},
      YEAR = {1987},
    NUMBER = {4},
     PAGES = {657--666},
      ISSN = {0002-9939,1088-6826},
   MRCLASS = {42B25 (42B20)},
  MRNUMBER = {877035},
MRREVIEWER = {Jos\'e\ L.\ Rubio de Francia},
     
}

@article {Ky2013PotentialAnal,
    AUTHOR = {Ky, L.D.},
     TITLE = {On {${\rm weak}^*$}-convergence in {$H^1_L(\Bbb{R}^d)$}},
   JOURNAL = {Potential Anal.},
    VOLUME = {39},
      YEAR = {2013},
    NUMBER = {4},
     PAGES = {355--368},
      ISSN = {0926-2601,1572-929X},
   MRCLASS = {42B35 (46E15)},
  MRNUMBER = {3116053},
MRREVIEWER = {Richard\ D.\ Carmichael},
}

@article{LL2011AdvMath,
	title = {{$BMO_L({\mathbb H}^n)$} spaces and {Carleson} measures for {Schr\"odinger} operators},
	journal = {Adv. Math.},
	volume = {228},
	number = {3},
	pages = {1631-1688},
	year = {2011},
	issn = {0001-8708},
	author = {Lin, C.C. and Liu, H.P.},
}

@article{LLS2025TokyoJMAth,
	author = {Lin, Q.Z. and Liu, H.S. and Song, L.},
	title = {The Behavior of Hardy--Littlewood Maximal Operator and Littlewood--Paley Operators on CMO Space},
	journal      = {Tokyo J. Math.},
	volume       = {48},
	number       = {1},
	pages        = {193--216},
	year       = {2025},
   }

@article {Shen1994IndianaUnivMathJ,
    AUTHOR = {Shen, Z.W.},
     TITLE = {On the {N}eumann problem for {S}chr\"odinger operators in
              {L}ipschitz domains},
   JOURNAL = {Indiana Univ. Math. J.},
  FJOURNAL = {Indiana University Mathematics Journal},
    VOLUME = {43},
      YEAR = {1994},
    NUMBER = {1},
     PAGES = {143--176},
      ISSN = {0022-2518,1943-5258},
   MRCLASS = {35J10},
  MRNUMBER = {1275456},
MRREVIEWER = {Karl-Josef\ Witsch},
}

@article {Shen1995AnnInstFourier,
    AUTHOR = {Shen, Z.W.},
     TITLE = {{$L^p$} estimates for {S}chr\"odinger operators with certain
              potentials},
   JOURNAL = {Ann. Inst. Fourier (Grenoble)},
  FJOURNAL = {Universit\'e{} de Grenoble. Annales de l'Institut Fourier},
    VOLUME = {45},
      YEAR = {1995},
    NUMBER = {2},
     PAGES = {513--546},
      ISSN = {0373-0956,1777-5310},
   MRCLASS = {35J10 (42B20)},
  MRNUMBER = {1343560},
MRREVIEWER = {V.\ S.\ Rabinovich},
       DOI = {10.5802/aif.1463},
       URL = {https://doi.org/10.5802/aif.1463},
}

@article{SW2022JGeomAnal,
	title={The {CMO--Dirichlet} Problem for the {Schr\"odinger} Equation in the {upper half-space and characterizations of CMO}},
	volume={32},
	ISSN={1559-002X},
	number={4},
	journal={J. Geom. Anal.},
	publisher={Springer Science and Business Media LLC},
pages={Paper No. 130, 37 pp.},
	author={Song, L. and Wu, L.C.},
	year={2022}, }

@incollection{DZ2002Book,
    author = {Dziuba{\'n}ski, J. and Zienkiewicz, J.},
    title ={{$H^p$} spaces for Schr\"odinger operators} ,
    booktitle ={Fourier Analysis and Related Topics} ,
    publisher ={ Banach Center Publications} ,
    year = {2002},
}

@book{Stein1970PrincetonUniversityPress,
  title={Singular integrals and differentiability properties of functions},
  author={Stein, E.M.},
   year={1970},
  publisher={No. 30. Princeton university press}
}

@book {Stein1993book,
    AUTHOR = {Stein, E.M.},
     TITLE = {Harmonic analysis: real-variable methods, orthogonality, and
              oscillatory integrals},
    VOLUME = {43},
 PUBLISHER = {Princeton University Press, Princeton, NJ},
      YEAR = {1993},
      ISBN = {0-691-03216-5},
   MRCLASS = {42-02 (35Sxx 43-02 47G30)},
  MRNUMBER = {1232192},
MRREVIEWER = {Michael\ Cowling},
}

@article {SY2016AdvMath,
    AUTHOR = {Song, L. and Yan, L.X.},
     TITLE = {A maximal function characterization for {H}ardy spaces
              associated to nonnegative self-adjoint operators satisfying
              {G}aussian estimates},
   JOURNAL = {Adv. Math.},
  FJOURNAL = {Advances in Mathematics},
    VOLUME = {287},
      YEAR = {2016},
     PAGES = {463--484},
      ISSN = {0001-8708,1090-2082},
   MRCLASS = {42B30 (42B35 47B38)},
  MRNUMBER = {3422683},
MRREVIEWER = {Yiyu\ Liang},
       DOI = {10.1016/j.aim.2015.09.026},
       URL = {https://doi.org/10.1016/j.aim.2015.09.026},
}

@article {AMR2008JGA,
    AUTHOR = {Auscher, P. and McIntosh, A. and Russ, E.},
     TITLE = {Hardy spaces of differential forms on {R}iemannian manifolds},
   JOURNAL = {J. Geom. Anal.},
  FJOURNAL = {Journal of Geometric Analysis},
    VOLUME = {18},
      YEAR = {2008},
    NUMBER = {1},
     PAGES = {192--248},
      ISSN = {1050-6926,1559-002X},
   MRCLASS = {42B30 (47A60 58A10)},
  MRNUMBER = {2365673},
MRREVIEWER = {Gilles\ Carron},
       DOI = {10.1007/s12220-007-9003-x},
       URL = {https://doi.org/10.1007/s12220-007-9003-x},
}

@article {LY1995ApproxTheoryAppl,
    AUTHOR = {Lu, S.Z. and Yang, D.C. },
     TITLE = {The central {BMO} spaces and {L}ittlewood-{P}aley operators},
   JOURNAL = {Approx. Theory Appl. (N.S.)},
  FJOURNAL = {Approximation Theory and its Applications. New Series},
    VOLUME = {11},
      YEAR = {1995},
    NUMBER = {3},
     PAGES = {72--94},
      ISSN = {1000-9221},
   MRCLASS = {42B25 (41A35 46E30)},
  MRNUMBER = {1370776},
MRREVIEWER = {James\ E.\ Daly},
}

@article {Sun2004NagoyaMathJ,
    AUTHOR = {Sun, Y.Z.},
     TITLE = {On the existence and boundedness of square function operators
              on {C}ampanato spaces},
   JOURNAL = {Nagoya Math. J.},
  FJOURNAL = {Nagoya Mathematical Journal},
    VOLUME = {173},
      YEAR = {2004},
     PAGES = {139--151},
      ISSN = {0027-7630,2152-6842},
   MRCLASS = {42B25 (42B30)},
  MRNUMBER = {2041758},
MRREVIEWER = {Paolo\ Boggiatto},
       DOI = {10.1017/S0027763000008746},
       URL = {https://doi.org/10.1017/S0027763000008746},
}

@article {WC1990ChineseAnnMath,
    AUTHOR = {Wang, S.L. and Chen, J.C.},
     TITLE = {Some notes on square function operators},
   JOURNAL = {(Chinese) Chinese Ann. Math. Ser. A},
  FJOURNAL = {Chinese Annals of Mathematics. Series A. Shuxue Niankan. A Ji},
    VOLUME = {11},
      YEAR = {1990},
    NUMBER = {5},
     PAGES = {630--638},
      ISSN = {1000-8314},
   MRCLASS = {42B20 (47G10)},
  MRNUMBER = {1092055},
MRREVIEWER = {Weiyi\ Su},
}

@article {DL2013JFA,
    AUTHOR = {Duong, X.T. and Li, J.},
     TITLE = {Hardy spaces associated to operators satisfying
              {D}avies-{G}affney estimates and bounded holomorphic
              functional calculus},
   JOURNAL = {J. Funct. Anal.},
  FJOURNAL = {Journal of Functional Analysis},
    VOLUME = {264},
      YEAR = {2013},
    NUMBER = {6},
     PAGES = {1409--1437},
      ISSN = {0022-1236,1096-0783},
   MRCLASS = {42B30 (42C15)},
  MRNUMBER = {3017269},
MRREVIEWER = {Oscar\ Blasco},
       DOI = {10.1016/j.jfa.2013.01.006},
       URL = {https://doi.org/10.1016/j.jfa.2013.01.006},
}

@incollection {Wang1985book,
    AUTHOR = {Wang, S.L.},
     TITLE = {Some properties of the {L}ittlewood-{P}aley {$g$}-function},
 BOOKTITLE = {Classical real analysis ({M}adison, {W}is., 1982)},
    SERIES = {Contemp. Math.},
    VOLUME = {42},
     PAGES = {191--202},
 PUBLISHER = {Amer. Math. Soc., Providence, RI},
      YEAR = {1985},
      ISBN = {0-8218-5045-8},
   MRCLASS = {42B25},
  MRNUMBER = {807991},
MRREVIEWER = {Douglas\ Kurtz},
       DOI = {10.1090/conm/042/807991},
       URL = {https://doi.org/10.1090/conm/042/807991},
}

@article{DY2005CPAM,
  author  = {Duong, X.T. and Yan, L.X.},
  title   = {New function spaces of {BMO} type, the {John--Nirenberg} inequality, interpolation, and applications},
  journal = {Comm. Pure Appl. Math. },
  volume  = {58},
  year    = {2005},
  number  = {10},
  pages   = {1375--1420}
}

@article{DY2005JAMS,
  author  = {Duong, X.T. and Yan, L.X.},
  title   = {Duality of {Hardy} and {BMO} spaces associated with operators with heat kernel bounds},
  journal = {J. Amer. Math. Soc.},
  volume  = {18},
  year    = {2005},
  number  = {4},
  pages   = {943--973}
}

@book{HLMMY2011MAMS,
  author    = {Hofmann, S. and Lu, G.Z. and Mitrea, D. and Mitrea, M. and Yan, L.X.},
  title     = {Hardy spaces associated to non-negative self-adjoint operators satisfying {Davies--Gaffney} estimates},
  series    = {Mem. Amer. Math. Soc.},
  volume    = {214},
YEAR = {2011},
    NUMBER = {1007},
     PAGES = {vi+78},
}

@article {Neri1975Studiamath,
    AUTHOR = {Neri, U.},
     TITLE = {Fractional integration on the space {$H\sp{1}$} and its dual},
   JOURNAL = {Studia Math.},
  FJOURNAL = {Polska Akademia Nauk. Instytut Matematyczny. Studia
              Mathematica},
    VOLUME = {53},
      YEAR = {1975},
    NUMBER = {2},
     PAGES = {175--189},
      ISSN = {0039-3223,1730-6337},
   MRCLASS = {46E35},
  MRNUMBER = {388074},
MRREVIEWER = {H.\ Triebel},
}

@article{HM2009MathAnn,
  author  = {Hofmann, S. and Mayboroda, S.},
  title   = {Hardy and {BMO} spaces associated to divergence form elliptic operators},
  journal = {Mathematische Annalen},
  volume  = {344},
  year    = {2009},
  number  = {1},
  pages   = {37--116}
}

@book{Auscher2007MAMS,
  author    = {Auscher, P.},
  title     = {On necessary and sufficient conditions for {$L^p$}-estimates of {Riesz} transforms associated to elliptic operators on {$\mathbb{R}^n$} and related estimates},
  series    = {Memoirs of the American Mathematical Society},
  volume    = {186},
  number    = {871},
  year      = {2007},
  pages     = {xviii+75}
}

@article {Uchiyama1978TohokuMathJ,
    AUTHOR = {Uchiyama, A.},
     TITLE = {On the compactness of operators of {H}ankel type},
   JOURNAL = {Tohoku Math. J. (2)},
  FJOURNAL = {The Tohoku Mathematical Journal. Second Series},
    VOLUME = {30},
      YEAR = {1978},
    NUMBER = {1},
     PAGES = {163--171},
      ISSN = {0040-8735,2186-585X},
   MRCLASS = {47B37 (44A25)},
  MRNUMBER = {467384},
MRREVIEWER = {Richard\ Rochberg},
}

\end{document}